\documentclass[11pt]{article}
\usepackage{placeins}
\usepackage{amsmath,amssymb,amsthm,mathtools}
\usepackage[margin=2.7cm]{geometry}
\usepackage{booktabs}
\usepackage{graphicx}
\usepackage{enumitem}
\usepackage{algorithm,algpseudocode}
\usepackage{hyperref}

\theoremstyle{plain}
\newtheorem{theorem}{Theorem}[section]
\newtheorem{proposition}[theorem]{Proposition}
\newtheorem{lemma}[theorem]{Lemma}
\newtheorem{corollary}[theorem]{Corollary}
\newtheorem*{mainthm}{Main Theorem}
\theoremstyle{definition}
\newtheorem{definition}[theorem]{Definition}

\newtheorem{remark}[theorem]{Remark}

\newcommand{\C}{\mathbb{C}}
\newcommand{\R}{\mathbb{R}}
\newcommand{\spec}{\operatorname{spec}}
\newcommand{\tr}{\operatorname{tr}}
\newcommand{\ran}{\operatorname{ran}}
\newcommand{\rank}{\operatorname{rank}}
\newcommand{\fl}{\operatorname{fl}}
\newcommand{\sig}{\sigma}
\newcommand{\sv}{s}
\newcommand{\sigmin}{\sigma_{\min}}

\DeclarePairedDelimiter{\norm}{\lVert}{\rVert}
\DeclarePairedDelimiter{\abs}{\lvert}{\rvert}
\newcommand{\svdbox}{\texttt{svdbox}}
\newcommand{\tnorm}[1]{\ensuremath{|\!|\!|#1|\!|\!|}}
\newcommand{\Qs}{\mathcal{Q}}

\title{The reverse Ritz inclusion: computable bounds on the pseudospectra of an operator by
those of a finite-rank compression}
\author{Isaia Nisoli\thanks{Instituto de Matem\'atica, Universidade Federal do Rio de Janeiro,
Av.\ Athos da Silveira Ramos 149, Centro de Tecnologia, Bloco C, Cidade Universit\'aria, Rio de
Janeiro, Brazil; Department of Mathematics, Hokkaido University, N10 W8, Kita-ku, Sapporo 001-0010,
Japan; RIES, Hokkaido University, N20 W10, Kita-ku, Sapporo 001-0020, Japan.
\texttt{nisoli@im.ufrj.br}; ORCID \href{https://orcid.org/0000-0003-2474-594X}{0000-0003-2474-594X}.}}
\date{}

\begin{document}
\maketitle

\begin{abstract}
We consider a bounded operator on a separable complex Hilbert space, reached only through its action
on vectors and that of its adjoint, and split it along a subspace of finite dimension into a
compression, a far block acting on the complement, and the two couplings between them. Given upper
bounds on the norms of the couplings and of the far block, we show that outside a disc a
pseudospectrum of the operator lies in a pseudospectrum of the compression, with an explicit loss in
the level; this reverses the classical inclusion of the Ritz values in a pseudospectrum of the
operator. The same bounds, with the smallest singular value of the shifted compression, bound the
resolvent of the operator, count its eigenvalues with algebraic multiplicity inside every contour
outside the disc on which a computed inequality holds, and enclose its spectral projector; no normality,
self-adjointness or positivity is used.
\end{abstract}

%=======================================================================
\section{Introduction}\label{sec:intro}
%=======================================================================

A large operator reached only through its application to vectors is usually reduced to a compression: a
handful of eigenvectors or Krylov vectors, assembled into a small matrix whose eigenvalues can be
computed. What the compression omits is not under the analyst's control, and for a non-normal operator
an eigenvalue of the compression need not sit near any eigenvalue of the operator itself, the gap
being invisible to the compression alone. We build the kept subspace as a frame joining
eigendirections to singular vectors; a bound on what the compression omits, together with the two
couplings between the kept subspace and the rest, then encloses the operator's eigenvalues outside a
disc, their count and the invariant subspace they span, through those of the compression.

Most of the known results relating the two spectra go from the compression to the operator. For a
self-adjoint operator, the distance from an eigenvalue of a compression to the spectrum of the
operator is at most the norm of the coupling from the kept subspace to its
complement~\cite[Lemma~1]{DaviesPlum2004}. For a general operator $A$, the Ritz values computed by $k$
steps of the Arnoldi process lie in the $\varepsilon$-pseudospectrum $\Lambda_\varepsilon(A)$, defined
in~\eqref{eq:pseudo} below, for every $\varepsilon>\abs{h_{k+1,k}}$, the last coefficient of the
recurrence~\cite[Thm.~26.1]{TrefethenEmbree2005}, and beyond this and their containment in the
numerical range little is known about the Ritz values of a non-Hermitian matrix~\cite{CardenEmbree2012};
for a non-normal operator $\Lambda_\varepsilon(A)$ can extend far from the spectrum, so this says
nothing about how far a Ritz value sits from an eigenvalue, and nothing is counted. Bounds in the
other direction are written through the distance from the kept subspace to an unknown eigenvector or
invariant subspace~\cite[Thm.~4.3]{Saad2011}, \cite[Thm.~4.1]{JiaStewart2000}, or through the
condition number of an unknown matrix of eigenvectors~\cite{BauerFike1960}, \cite[Thm.~3.6]{Saad2011};
for a non-normal matrix the residual norms of approximate eigenvectors do not by themselves bound the
error in the eigenvalues~\cite{KahanEtAl1982}. The other direction does hold at a fixed compression, for a
closed operator on a Hilbert space as for a matrix, when the product of the two couplings is small
against the square of the separation between the compression and the part of the operator acting on
the complement of the kept subspace: the operator then has an invariant subspace near the kept one,
on which it acts as the compression plus a perturbation bounded through that separation, and its
spectrum splits into the $k$ eigenvalues of that $k\times k$ matrix and the
rest~\cite[Thm.~3.5]{Stewart1971}, \cite[Thms.~4.1 and~4.12]{Stewart1973},
\cite[Thm.~V.2.7]{StewartSun1990}. The separation is the reciprocal of the
norm of the inverse of a Sylvester map built from the compression and the part acting on the
complement, and so involves the whole of that part; Stewart notes that this inverse will generally be difficult to
evaluate~\cite[\S3]{Stewart1973}. In the Rayleigh--Ritz setting this hypothesis is not automatically satisfied, and the
separation can be much smaller than the distance between the two spectra, so that it has to be
assumed~\cite[\S\S2 and~5]{JiaStewart2000}. For a sequence of approximations of a non-self-adjoint operator converging in a suitable
sense, each eigenvalue of algebraic multiplicity $m$ is matched by exactly $m$ eigenvalues of the
approximation~\cite[Thm.~2.6]{Boegli2017}, and each point of $\Lambda_\varepsilon$ of the operator
lies in $\Lambda_\varepsilon$ of the approximation~\cite[Thm.~3.3]{Boegli2018}, both from some index
on that the theorems do not compute.

The conclusions of our result are of the same kind, the eigenvalues outside a disc located through the compression and
counted, and the invariant subspace they span enclosed through its spectral projector, while its
hypothesis is an inequality verified on a contour, in which the part acting on the complement enters
only through an upper bound on its norm, computed in dimension $n$ from $mn$ applications of the operator at
order $m$ of the trace bound, where no deterministic method making fewer than $n-k$ applications
returns a certified upper bound. The separation involves every eigenvalue of the compression, including those inside the
disc, where they may lie close to the spectrum of the part acting on the complement; in our hypothesis
the eigenvalues of the compression inside the disc enter only through $\sigmin$ of the shifted
compression on the contour, so they are allowed anywhere inside it. Our reduction needs this: the frame
of Definition~\ref{def:Q} joins the eigendirections outside the disc to $k_s$ left and $k_s$ right
singular vectors, with $k_s$ chosen to bring $\norm{G_{22}}$ below the radius and not tied to the
number of eigenvalues sought, so the compression has up to $2k_s$ further eigenvalues, which take no
part in the statements as long as they stay inside the disc.

Let $\mathcal H$ be a separable complex Hilbert space, $A$ a bounded operator on it and $Q:\C^k\to\mathcal
H$ an isometry with range $\Qs$. Completing $Q$ to an orthonormal basis $[\,Q\mid Q_\perp\,]$ writes
\begin{equation}\label{eq:introblocks}
  A\;\cong\;\begin{pmatrix} G_{11} & G_{12}\\ G_{21} & G_{22}\end{pmatrix},
  \qquad
  \begin{aligned}
    G_{11}&=Q^{*}AQ, & G_{12}&=Q^{*}AQ_\perp,\\
    G_{21}&=Q_\perp^{*}AQ, & G_{22}&=Q_\perp^{*}AQ_\perp,
  \end{aligned}
\end{equation}
a unitary similarity. With $P_\perp=I-QQ^{*}$ one has $\norm{G_{12}}=\norm{Q^{*}AP_\perp}$,
$\norm{G_{21}}=\norm{P_\perp AQ}$ and $\norm{G_{22}}=\norm{P_\perp AP_\perp}$, so the three norms do
not depend on the choice of $Q_\perp$. We write $\norm\cdot$ for the norm of a vector and the operator
norm it induces, $\kappa(W)=\norm W\norm{W^{-1}}$, and
\begin{equation}\label{eq:jintro}
  j(T)=\inf_{\norm x=1}\norm{Tx},
\end{equation}
the injection modulus, equal to $\sigmin(T)$ in finite dimension and to $1/\norm{T^{-1}}$ whenever $T$
is invertible. The $\varepsilon$-pseudospectrum is
\begin{equation}\label{eq:pseudo}
  \Lambda_\varepsilon(M)=\{z\in\C:\norm{(zI-M)^{-1}}\ge1/\varepsilon\},
\end{equation}
with $\norm{(zI-M)^{-1}}=\infty$ on the spectrum, so $\spec(M)\subseteq\Lambda_\varepsilon(M)$ for
every $\varepsilon>0$~\cite{TrefethenEmbree2005}. In finite dimension
$\Lambda_\varepsilon(M)=\{z:j(M-zI)\le\varepsilon\}$; in general the set on the right can be the
smaller of the two, for the reason given in Section~\ref{sec:bound}.

\subsection{The certificate}\label{sec:intromain}

\begin{mainthm}
Let $A$ be a bounded operator on a separable complex Hilbert space $\mathcal H$, let $Q:\C^k\to\mathcal H$
be an isometry, let the blocks be as in~\eqref{eq:introblocks}, and let $\bar
g,\gamma_1,\gamma_2$ satisfy
\[
  \norm{G_{22}}\le\bar g,\qquad \norm{G_{12}}\le\gamma_1,\qquad \norm{G_{21}}\le\gamma_2 .
\]
Write $\sig_{11}(z)=\sigmin(G_{11}-zI)$. Then at every $z$ with $\abs{z}>\bar g$ and
$\sig_{11}(z)(\abs{z}-\bar g)>\gamma_1\gamma_2$,
\begin{equation}\label{eq:introres}
  \norm{(A-zI)^{-1}}\;\le\;
  \frac{\sig_{11}(z)+(\abs{z}-\bar g)+\gamma_1+\gamma_2}
       {\sig_{11}(z)\,(\abs{z}-\bar g)-\gamma_1\gamma_2}\,;
\end{equation}
for every $R>\bar g$ and $\varepsilon>0$ with $\varepsilon(R-\bar g)>\gamma_1\gamma_2$,
\begin{equation}\label{eq:introincl}
  \Lambda_{\delta(\varepsilon)}(A)\cap\{\abs{z}>R\}\;\subseteq\;
  \Lambda_{\varepsilon}(G_{11})\cap\{\abs{z}>R\},\qquad
  \delta(\varepsilon)=\frac{\varepsilon(R-\bar g)-\gamma_1\gamma_2}
                           {\varepsilon+(R-\bar g)+\gamma_1+\gamma_2}\,;
\end{equation}
and if $\sig_{11}(z)(\abs{z}-\bar g)>\gamma_1\gamma_2$ at every $z$ of the circle $\abs{z}=R$, then
$A$ and the $k\times k$ matrix $G_{11}$ have the same number of eigenvalues, counted with algebraic
multiplicity and finite in number, in $\{\abs{z}>R\}$.
\end{mainthm}

In particular an eigenvalue $\lambda$ of $A$ with $\abs\lambda>\bar g$ satisfies
$\sig_{11}(\lambda)(\abs\lambda-\bar g)\le\gamma_1\gamma_2$, since otherwise \eqref{eq:introres}
would make $A-\lambda I$ invertible: $\lambda$ lies in $\Lambda_\varepsilon(G_{11})$ for
$\varepsilon=\gamma_1\gamma_2/(\abs\lambda-\bar g)$ when this is positive, and is an eigenvalue of
$G_{11}$ when $\gamma_1\gamma_2=0$.

The Main Theorem is proved in Section~\ref{sec:bound}, which also says how far it reaches.

The Main Theorem holds for every frame $Q$, so it does not say which one to use. What we want is the
resolvent of $A$ outside some radius $R$, together with the eigenvalues there and their count, and
the frame we build for that joins eigendirections with singular directions: the Arnoldi eigenvectors with
$\abs\lambda>R$, approximate eigenvectors whose span is close to an invariant subspace when the
Arnoldi process has converged, and the leading left and right singular vectors of $A$, aimed at
bringing $\norm{G_{22}}$ below $R$ (Proposition~\ref{prop:frame}; Section~\ref{sec:frame}).
Building the frame this way is a search, uncertified; what is certified is the hypothesis, checked
afterwards on the frame the search returns.

\subsection{The inputs}\label{sec:introinputs}

The four quantities the Main Theorem uses are of three kinds.

The compression $G_{11}$ is $k\times k$ regardless of the
dimension of $\mathcal H$, so one verified singular value decomposition per $z$ encloses $\sig_{11}(z)$
(Section~\ref{sec:dense}, Section~\ref{sec:vbd}). The couplings have rank at most $k$ but map into
$\mathcal H\ominus\Qs$, which need not be finite-dimensional, so no factorisation of the coupling
itself is available in general; $\gamma_1,\gamma_2$ come instead from two Gram matrices of size $k$,
formed once per operator at $2k$ applications of the operator (Section~\ref{sec:coupling}).

The bound $\bar g\ge\norm{G_{22}}$ is the only quantity
that does not live on an object of low rank. In finite dimension we compute it by a sweep that
applies the operator $m$ times per column at order $m$, runs in floating point, and is certified afterwards by two
scalar inequalities (Section~\ref{sec:norm}).

On an infinite-dimensional space no finite number of applications certifies $\bar g$ (the argument of
Proposition~\ref{prop:floor} applies to any finite number of applications), and it has to be assumed.

A certified number is an outward-rounded ball enclosure, in the sense of validated
numerics~\cite{Tucker2011}: a floating-point phase produces a candidate
and a second phase in ball arithmetic bounds its residual, as in~\cite{Rump2022a,Miyajima2014a} and
with the ball arithmetic of BallArithmetic.jl~\cite{BallArithmetic2023}. Where a bound with explicit
constants is available we use that instead, as for the rounding of the sweep of
Section~\ref{sec:fpsweep}.

\subsection{Origin and reach}\label{sec:origin}

The certificate uses no normality, self-adjointness or positivity, and holds for any bounded operator
on a separable complex Hilbert space, reached only through $x\mapsto Ax$ and $x\mapsto A^{*}x$; cost is counted in applications of $A$
and of $A^{*}$, a stored matrix being the case in which each application is a multiplication. For a
self-adjoint operator it specialises to an enclosure of the extreme eigenvalues with a radius
quadratic in the coupling, and to their count (Corollary~\ref{cor:selfadj}).

The question that started this is narrower: the discretisations of Markov and transfer operators, too
large to store, of which one wants the leading eigenvalues, the mixing rate and the peripheral
spectrum. Certifying the isolated eigenvalues of such an operator from a finite-dimensional
discretisation, with a computable approximation bound, is carried out
in~\cite{BlumenthalEtAl2025,Nisoli2026}; what that argument uses, and what this paper supplies, is
a lower bound on $\sigmin(A-zI)$ for a discretisation that is large and strongly non-normal, in the
Euclidean norm~\cite{Nisoli2026} or in a norm bounded through it~\cite{BlumenthalEtAl2025}.
Section~\ref{sec:numbz} certifies such a bound, with the count and isolation discs of the leading
eigenvalues, for the Ulam discretisation of a noisy Belousov--Zhabotinsky map at $2^{16}$ cells,
reached only through its applications to vectors. A
concrete instance is the statistical periodicity of noise-induced order, where the period is set by
the leading nontrivial resonance of an annealed transfer operator~\cite{SatoNisoli2026}; that the
resonance is the leading one is a claim about what is not there, which the count of
Section~\ref{sec:exterior} supplies for a discretisation.
For the transfer operator itself, of an analytic expanding map of a domain in $\C^d$ acting on
holomorphic functions, the singular values are bounded by $C\exp(-cn^{1/d})$ with $C$ and $c$ computed
from the map~\cite{Bandtlow2008,BandtlowJenkinson2008}. By Proposition~\ref{prop:frame} such a bound
gives $\bar g$ for a frame containing the exact leading singular vectors on one side; for a computed
frame a bound on the distance between the computed and the exact singular subspaces is needed in
addition, and we do not supply it.

\subsection{Related work}\label{sec:related}

The invertibility half of Theorem~\ref{thm:assembly} is known: for a bounded $2\times2$ block
operator, $\lambda$ lies in the resolvent set whenever both diagonal blocks are invertible at
$\lambda$ and $\norm{(A_{11}-\lambda)^{-1}A_{12}(A_{22}-\lambda)^{-1}A_{21}}<1$, together with the
same inequality with the indices exchanged~\cite[Prop.~1.13.3]{Tretter2008}, and
$\sig_{11}\sig_{22}>\gamma_1\gamma_2$ implies both. The product condition is Ostrowski's, and
Brauer deduced from it the ovals of Cassini~\cite{Brauer1947},~\cite[\S2.1]{Varga2004}; and the block Gershgorin sets built from norms of inverses
are~\cite[Thm.~1.13.1]{Tretter2008} and~\cite[Ch.~6]{Varga2004}. The quadratic numerical range
$W^2$, the union of the spectra of the $2\times2$ matrices of the quadratic forms of the four blocks,
contains the spectrum, bounds the resolvent through the distance to it, and, when its closure has two
components containing the numerical ranges of the two diagonal blocks, splits the spectrum
accordingly, the part in the first component being that of an operator on the first
space~\cite[Thms.~2.3, 3.1, 4.4 and~5.1]{LangerEtAl2001}; its definition runs over the quadratic forms
of the far block on the whole of its space.

Stewart's theorems hold for non-normal matrices and closed
operators~\cite{Stewart1971,Stewart1973}, but their hypothesis is on the separation, which he
evaluates only for Hermitian blocks, where in the Frobenius norm it equals the distance between the
two spectra~\cite[Thm.~4.7]{Stewart1973}; for diagonalisable blocks he bounds it through that
distance and the condition numbers of the two eigenvector matrices~\cite[Thm.~4.8 and Cor.~4.9]{Stewart1973}, one
of which belongs to the part acting on the complement, and he notes that the separation may be much
smaller than the distance between the spectra and that its estimation is a difficult
problem~\cite[\S4.3]{Stewart1973}. The integral representation of the inverse of the Sylvester
map~\cite[(1.2)]{LumerRosenblum1959} bounds it below through resolvent norms on a contour separating
the two spectra, and with $\norm{(G_{22}-zI)^{-1}}\le(\abs z-\bar g)^{-1}$ this makes the hypothesis
checkable from the four numbers of the Main Theorem, but only when the whole spectrum of $G_{11}$ lies
outside a contour in $\{\abs z>\bar g\}$. The frame of Definition~\ref{def:Q} puts further
eigenvalues into the compression, inside the disc, so Stewart's hypothesis cannot be
checked this way on the compressions this paper builds, while the Main Theorem allows those
eigenvalues anywhere inside it.

What this paper adds to these is, first, a construction of the frame to be certified, joining the eigendirections outside the disc
to left and right singular vectors so as to bring $\bar g$ below the radius and to make the
coupling small, with both $\gamma_1$ and $\gamma_2$ bounded by the first singular value of $A$
left out of the frame (Section~\ref{sec:frame}); then a hypothesis checked on a contour, under which the further
eigenvalues this frame brings into the compression may lie anywhere inside the disc; the explicit
constant in~\eqref{eq:assembly}, with the count and the projector; and the certified
bound on $\norm{G_{22}}$, at $mn$ applications of the operator against a lower bound of $n-k$ for any
deterministic method.

The existing routes to an exterior count each want something that an operator reached only through
$x\mapsto Ax$ and $x\mapsto A^{*}x$ does not provide: the inertia
of $A-\sigma I$ wants $A=A^{*}$ and a factorisation of the large matrix; the variational enclosures
of~\cite{Lehmann1963,Behnke1991,Plum1991} want a bound on the next eigenvalue, which is the quantity
to be produced; the contour methods~\cite{SakuraiSugiura2003,Beyn2012}, verified
in~\cite{ImakuraEtAl2020}, want a certified solve at every quadrature node. The method
of~\cite{TeraoEtAl2025} bounds the smallest singular value of a matrix that is not Hermitian by
testing the positive definiteness of $A^{*}A-\theta I$ and $AA^{*}-\theta I$ at each point; it needs a
factorisation of these matrices, which are formed from the stored matrix, and for this reason we did
not compare it with ours.

Reductions that build a coarse operator from the leading eigenvectors and ask that the leading
spectrum survive~\cite{DeuflhardEtAl2000,DeuflhardWeber2005,FroylandDellnitz2003} carry a priori
bounds on the eigenvalue error~\cite{SarichEtAl2010,DjurdjevacEtAl2012}, written through the
projection error of the true eigenfunctions, which is not computable without the answer; ours are a
posteriori, every quantity being computed and outward rounded. The price is that $Q$ is not a
partition and $G_{11}$ is not a chain on aggregated states.

Extended dynamic mode decomposition builds a Galerkin matrix for the Koopman operator on a dictionary
of observables from snapshot pairs~\cite{WilliamsEtAl2015}; residual dynamic mode decomposition
attaches to each candidate the Galerkin residual over the dictionary and returns, for each
$\varepsilon$, a set contained in the approximate point pseudospectrum and converging to it as the
dictionary grows~\cite{ColbrookTownsend2024,Colbrook2024}; the adjoint side, the transfer operator
itself, and spaces other than $L^2$ are treated in~\cite{HerwigEtAl2025}. The quantity is the one the
Main Theorem turns on, used in the other direction: for the Koopman operator $K$, a residual
$\norm{(K-\lambda I)g}/\norm{g}$ bounds
$\sigmin(K-\lambda I)$ from above, which places the set returned inside the pseudospectrum, while
\eqref{eq:introincl} bounds $\sigmin$ from below and so states that outside the disc there is nothing
further to find. Their statements are about the operator; ours are about a discretisation of it, and
a bound on a discretisation says nothing about the operator without the argument
of~\cite{BlumenthalEtAl2025,Nisoli2026} recalled in Section~\ref{sec:origin}, which this paper does
not carry out. Theorem~5.1
of~\cite{ColbrookTownsend2024} has the finite section resolvent of an isometry converging at every
$\abs\lambda>1$, by a Neumann series in which $\norm K=1$ plays the part $\bar g$ plays in
\eqref{eq:introres}, known in advance there and certified here.

%=======================================================================
\section{The two-block bound}\label{sec:bound}
%=======================================================================

Throughout this section $\mathcal H$ is a separable complex Hilbert space, $A$ a bounded operator on it,
$Q:\C^k\to\mathcal H$ an isometry, $P_\perp=I-QQ^{*}$, and the blocks are those of
\eqref{eq:introblocks}.

Every statement is in the resolvent norm, and a bound on $\norm{(T-zI)^{-1}}$ presupposes the
inverse. What a lower estimate of $\norm{Tx}$ produces is the injection modulus \eqref{eq:jintro}, and
in finite dimension $j(T)>0$ gives invertibility; for an operator it gives injectivity with closed
range, the unilateral shift on $\ell^2(\mathbb N)$ having $j(S)=1$ and no inverse. What is equivalent
is
\begin{equation}\label{eq:invcrit}
  j(T)>0\quad\text{and}\quad j(T^{*})>0,
\end{equation}
the first making the range closed and the second, through $\ker T^{*}=(\ran T)^{\perp}$,
dense~\cite[III-(5.10), IV-Thm.~5.2]{Kato1995}. The proofs below bound both, and once $T$ is
invertible $\norm{T^{-1}}=1/j(T)$. It is $j$ and not $\sigmin$ because a general bounded operator has
no singular value sequence, only the spectrum of $\abs T=(T^{*}T)^{1/2}$, of which $j(T)$ is the
bottom; the same distinction makes $\{z:j(M-zI)\le\varepsilon\}$ a subset of
$\Lambda_\varepsilon(M)$, the difference carried by the eigenvalues of the
adjoint~\cite[App.~B]{ColbrookTownsend2024}.

\subsection{The floor}\label{sec:assembly}

\begin{theorem}\label{thm:assembly}
Let $z\in\C$ and let $\gamma_1\ge\norm{G_{12}}$, $\gamma_2\ge\norm{G_{21}}$. Suppose the two diagonal
blocks are invertible on their spaces at $z$, and write
\[
\sig_{11}(z)=\frac1{\norm{(G_{11}-zI)^{-1}}}=\sigmin(G_{11}-zI),\qquad
\sig_{22}(z)=\frac1{\norm{(G_{22}-zI)^{-1}}} .
\]
If $\sig_{11}(z)\,\sig_{22}(z)>\gamma_1\gamma_2$ then $A-zI$ is invertible and
\begin{equation}\label{eq:assembly}
  \frac1{\norm{(A-zI)^{-1}}}\;\ge\;
  \frac{\sig_{11}(z)\,\sig_{22}(z)-\gamma_1\gamma_2}
       {\sig_{11}(z)+\sig_{22}(z)+\gamma_1+\gamma_2}\;>\;0 ,
\end{equation}
the right-hand side increasing in each of $\sig_{11},\sig_{22}$ and decreasing in each of
$\gamma_1,\gamma_2$. Only the two diagonal quantities depend on $z$.
\end{theorem}

\begin{proof}
Write a unit vector as $x=Qx_1+x_2$ with $x_1\in\C^k$, $x_2\in\ran P_\perp$, $a=\norm{x_1}$,
$b=\norm{x_2}$, $a^2+b^2=1$, and $t=\norm{(A-zI)x}$. The two components of $(A-zI)x$ are orthogonal,
so $t$ is at least the norm of either, whence $\sig_{11}a-\gamma_1b\le t$ and
$\sig_{22}b-\gamma_2a\le t$. From the first, $a\le(t+\gamma_1b)/\sig_{11}$; substituting into the
second, $b(\sig_{11}\sig_{22}-\gamma_1\gamma_2)\le t(\sig_{11}+\gamma_2)$, and symmetrically
$a(\sig_{11}\sig_{22}-\gamma_1\gamma_2)\le t(\sig_{22}+\gamma_1)$. Add, and use
$a+b\ge\sqrt{a^2+b^2}=1$. Monotonicity is the derivative in $\sig_{11}$, whose numerator is
$\sig_{22}^2+\sig_{22}(\gamma_1+\gamma_2)+\gamma_1\gamma_2>0$. That bounds $j(A-zI)$. The adjoint has
in the same splitting the diagonal blocks $(G_{11}-zI)^{*}$ and $(G_{22}-zI)^{*}$, invertible with
the same resolvent norms, and the off-diagonal blocks $G_{21}^{*}$, $G_{12}^{*}$ of norms $\gamma_2$
and $\gamma_1$, so the same computation bounds $j\bigl((A-zI)^{*}\bigr)$ by the same quantity; apply
\eqref{eq:invcrit}.
\end{proof}

\begin{remark}\label{rem:counterexample}
For $A=\bigl(\begin{smallmatrix}1&1\\1&1\end{smallmatrix}\bigr)$ with $\Qs=\operatorname{span}e_1$
and $z=0$, all four quantities are $1$, so $\sig_{11}\sig_{22}-\gamma_1\gamma_2=0$, while
$\sigmin(A)=0$: no bound positive under $\sig_{11},\sig_{22}>0$ alone can hold. This $A$ is
symmetric, so what fails is the splitting and not the normality.
\end{remark}

\begin{proposition}\label{prop:blockdom}
In the setting of Theorem~\ref{thm:assembly}, if $\gamma_1<\sig_{11}(z)$ and
$\gamma_2<\sig_{22}(z)$, then $A-zI$ is invertible and
\begin{equation}\label{eq:blockdom}
  \frac1{\norm{(A-zI)^{-1}}}\;\ge\;\min\bigl(\sig_{11}(z),\sig_{22}(z)\bigr)
  \Bigl(1-\max\Bigl(\frac{\gamma_1}{\sig_{11}(z)},\frac{\gamma_2}{\sig_{22}(z)}\Bigr)\Bigr).
\end{equation}
\end{proposition}

\begin{proof}
Write $A-zI=D(I+D^{-1}E)$ with $D$ block diagonal and $E$ the off-diagonal part. The blocks of
$D^{-1}E$ have norms at most $\gamma_1/\sig_{11}$ and $\gamma_2/\sig_{22}$, and the operator is block
anti-diagonal, so $\norm{D^{-1}E}$ is at most the larger of the two ratios, which is below $1$; the Neumann series
gives $\norm{(A-zI)^{-1}}\le\norm{D^{-1}}/(1-\norm{D^{-1}E})$ with
$\norm{D^{-1}}=1/\min(\sig_{11},\sig_{22})$.
\end{proof}

\begin{remark}
Proposition~\ref{prop:blockdom} is tighter than \eqref{eq:assembly} where the couplings are small
relative to the gaps, and its hypothesis fails as $z$ approaches an eigenvalue of $G_{11}$, where
the hypothesis of \eqref{eq:assembly} holds as long as $\sig_{11}>\gamma_1\gamma_2/\sig_{22}$. Its hypothesis
implies $\sig_{11}\sig_{22}>\gamma_1\gamma_2$, so it never applies where \eqref{eq:assembly} does
not.
\end{remark}

\subsection{The far block}\label{sec:power}

The shift enters the far block diagonally, $P_\perp(A-zI)P_\perp=G_{22}-zI$, so one $z$-independent
bound serves at every $z$: if $\bar g\ge\norm{G_{22}}$ and $\abs z>\bar g$ then a Neumann series
gives $\norm{(G_{22}-zI)^{-1}}\le(\abs z-\bar g)^{-1}$, that is
\begin{equation}\label{eq:farshort}
  \sig_{22}(z)\;\ge\;\abs{z}-\bar g .
\end{equation}

\subsection{Inclusion, count and projector}\label{sec:extbound}

\begin{proposition}\label{prop:complete}
Let $\bar g\ge\norm{G_{22}}$, $\gamma_1\ge\norm{G_{12}}$, $\gamma_2\ge\norm{G_{21}}$, and let
$R>\bar g$ and $\varepsilon>0$ satisfy $\varepsilon(R-\bar g)>\gamma_1\gamma_2$. Then
\begin{equation}\label{eq:incl}
  \Lambda_{\delta}(A)\cap\{\abs z>R\}\;\subseteq\;\Lambda_{\varepsilon}(G_{11})\cap\{\abs z>R\},
  \qquad
  \delta=\frac{\varepsilon(R-\bar g)-\gamma_1\gamma_2}
              {\varepsilon+(R-\bar g)+\gamma_1+\gamma_2}\;>\;0 .
\end{equation}
At a single $z$ the same computation with $\sig_{22}(z)$ in place of $R-\bar g$ gives a larger
$\delta$. An eigenvalue $\lambda$ of $A$ with $\abs\lambda>\bar g$, at which $G_{22}-\lambda I$ is
invertible by \eqref{eq:farshort}, satisfies $\sig_{11}(\lambda)\sig_{22}(\lambda)\le\gamma_1\gamma_2$,
so it is an eigenvalue of $G_{11}$ when $\gamma_1\gamma_2=0$.
\end{proposition}

\begin{proof}
Let $\abs z>R$ with $\sig_{11}(z)>\varepsilon$. By \eqref{eq:farshort} $\sig_{22}(z)>R-\bar g$, and
the right-hand side of \eqref{eq:assembly} increases in $\sig_{11}$ and $\sig_{22}$, so
Theorem~\ref{thm:assembly} gives $\norm{(A-zI)^{-1}}<1/\delta$. At a singular $\lambda$ the
hypothesis of Theorem~\ref{thm:assembly} fails.
\end{proof}

\begin{proposition}\label{prop:count}
Let $\bar g\ge\norm{G_{22}}$, $\gamma_1\ge\norm{G_{12}}$, $\gamma_2\ge\norm{G_{21}}$, let $\Gamma$ be a
rectifiable Jordan curve on which $\sig_{11}(z)\sig_{22}(z)>\gamma_1\gamma_2$ pointwise, and let $D$ be a connected component of
$\C\setminus\Gamma$ with $D\cup\Gamma\subseteq\{\abs z>\bar g\}$. Then $A$ and $G_{11}$ have the same
number of eigenvalues in $D$, counted with algebraic multiplicity, and that number is at most $k$.
Nothing is counted inside $\{\abs z\le\bar g\}$, where no hypothesis of this paper holds.
\end{proposition}

\begin{proof}
Put $A_t=\bigl(\begin{smallmatrix}G_{11}&tG_{12}\\ tG_{21}&G_{22}\end{smallmatrix}\bigr)$ for
$t\in[0,1]$. Its couplings have product at most $\gamma_1\gamma_2$ and \eqref{eq:assembly} decreases
in them, so Theorem~\ref{thm:assembly} bounds $\norm{(A_t-zI)^{-1}}$ on $\Gamma\times[0,1]$ by the
reciprocal of the right-hand side of \eqref{eq:assembly}, which is continuous and positive on the
compact $\Gamma$, $\sig_{11}$ and $\sig_{22}$ being $1$-Lipschitz, and so has a positive minimum. If $D$ is bounded let $\Gamma_D$ be its boundary; otherwise let
$r_0>\max(\sup_t\norm{A_t},\max_{z\in\Gamma}\abs z)$ and let $\Gamma_D$ be the boundary of
$D\cap\{\abs z<r_0\}$, which contains the
same spectrum since $\spec(A_t)\subseteq\{\abs z\le\norm{A_t}\}$ and the resolvent is bounded on the
added circle by a Neumann series. Then $P_t=\frac1{2\pi i}\oint_{\Gamma_D}(zI-A_t)^{-1}\,dz$ is a
projection commuting with $A_t$, the Riesz projection for the part of $\spec(A_t)$ in
$D$~\cite[III-\S6.4]{Kato1995}, and depends continuously on $t$ in norm, by the resolvent identity.
Two projections at norm distance less than one have ranges of the same
dimension~\cite[I-\S4.6 and Problem~III-3.21]{Kato1995}, so $\rank P_t$ is constant. At $t=0$ the
operator is block diagonal and $\spec(G_{22})\subseteq\{\abs z\le\bar g\}$ misses $D$, so $P_0$ is
the spectral projection of $G_{11}$ for its eigenvalues in $D$, of rank at most $k$. Hence every
$P_t$ has finite rank, and that rank is the number of eigenvalues of $A_t$ in $D$ counted with
algebraic multiplicity~\cite[III-\S6.5]{Kato1995}.
\end{proof}

\begin{corollary}\label{cor:agree}
Let $R>\bar g$, and let $D_1,\dots,D_p$ be disjoint closed discs in $\{\abs z>R\}$ whose union
contains every eigenvalue of $G_{11}$ of modulus $>R$. If the hypothesis of
Proposition~\ref{prop:count} holds on $\abs{z}=R$ and on each $\partial D_i$, then every eigenvalue
of $A$ of modulus $>R$ lies in $D_1\cup\dots\cup D_p$, and each $D_i$ contains as many eigenvalues of
$A$ as of $G_{11}$, counted with algebraic multiplicity.
\end{corollary}

\begin{proof}
Apply Proposition~\ref{prop:count} to the interior of each $D_i$ and to $\{\abs z>R\}$. The
hypothesis on $\partial D_i$ excludes eigenvalues of $A$ and of $G_{11}$ there, so each closed $D_i$
holds as many eigenvalues of $A$ as of $G_{11}$. Summing over $i$, the discs hold as many eigenvalues
of $A$ as $G_{11}$ has in $\{\abs z>R\}$, which is the number $A$ has there.
\end{proof}

\begin{remark}\label{rem:radii}
When $\gamma_1\gamma_2>0$ the discs cannot shrink to points, since
$\sig_{11}$ vanishes at the eigenvalues of $G_{11}$ and the hypothesis asks
$\sig_{11}>\gamma_1\gamma_2/\sig_{22}$ on each $\partial D_i$, so their radii are how well the
eigenvalues of $A$ are located. Inside the disc nothing is claimed: an eigenvalue of $G_{11}$ with
$\abs\lambda<R$ is neither asserted to be an eigenvalue of $A$ nor a fault of the frame, and it costs
something only through $\sig_{11}$, which it depresses in its neighbourhood, so an artefact just
inside the circle can make the hypothesis fail on the circle.
\end{remark}

\begin{remark}\label{rem:cover}
The hypothesis is checked at finitely many points of the contour. Both $\sig_{11}$ and $\abs z-\bar g$
are $1$-Lipschitz in $z$, so on an arc between two samples each is at least its certified lower bound
at the nearer sample less half the length of the arc; the sample points are computed in floating
point, and a margin bounding their distance from the exact points of the circle is subtracted as
well. Testing the hypothesis with these lower bounds, and bisecting the arcs on which it fails,
certifies it at every point of the circle. When the frame is not orthonormal the lower bound on
$\sig_{11}$ used here is the one of Proposition~\ref{prop:almostortho}, which bounds $\sigmin$ of an
exact compression and is again $1$-Lipschitz. The bisection is capped, at depth $64$ in
Section~\ref{sec:numerics}; a cover then fails only where the hypothesis fails at a sample point,
where an arc cannot be split in floating point, or where the cap is reached.
\end{remark}

\begin{proposition}\label{prop:projector}
Let $\Gamma$ and $D$ be as in Proposition~\ref{prop:count} and $\abs\Gamma$ the length of $\Gamma$.
Let $M$ be the supremum over $z\in\Gamma$ of the reciprocal of the right-hand side of
\eqref{eq:assembly}, and
\[
  M_0=\sup_{z\in\Gamma}\max\Bigl(\frac1{\sig_{11}(z)},\frac1{\sig_{22}(z)}\Bigr)\ \le\ M .
\]
Let $P$ be the Riesz projector of $A$ for its spectrum in $D$ and $P_0$ the spectral projector of
$G_{11}$ for its eigenvalues in $D$, padded by zero on $\ran P_\perp$. Then
\begin{equation}\label{eq:projector}
  \norm{P-P_0}\;\le\;\frac{\abs\Gamma}{2\pi}\,M\,M_0\,\max(\gamma_1,\gamma_2),
\end{equation}
which on the circle $\abs z=R$ reads $R\,M\,M_0\max(\gamma_1,\gamma_2)$. The range of $P_0$ is
$Q\,\ran P_{G_{11}}$, and the gap between the two ranges is at most $\norm{P-P_0}$.
\end{proposition}

\begin{proof}
Write $A=A_0+E$ with $A_0=G_{11}\oplus G_{22}$ and
$E=\bigl(\begin{smallmatrix}0&G_{12}\\ G_{21}&0\end{smallmatrix}\bigr)$, so $E^{*}E$ is block
diagonal and $\norm E=\max(\norm{G_{12}},\norm{G_{21}})$. On $\Gamma$ both operators are invertible
and the resolvent identity gives $(zI-A)^{-1}-(zI-A_0)^{-1}=(zI-A)^{-1}E(zI-A_0)^{-1}$; integrate
over $\Gamma$ and divide by $2\pi i$. When $D$ is unbounded the added circle of radius $M'$ carries
an integrand bounded by a constant times $\norm E/M'^{2}$, whose contribution vanishes. For the last claim, a unit
$x\in\ran P$ has $(I-P_0)x=(P-P_0)x$, and symmetrically.
\end{proof}

\begin{remark}\label{rem:projcontour}
The projector, and with it the invariant subspace of $A$, is enclosed by the resolvent bounds alone,
with no fixed-point argument such as those by which the eigenpair enclosures
of~\cite{Rump2022a,Miyajima2014a} are refined. The contour should be taken as large as
Proposition~\ref{prop:count} allows. For a semisimple eigenvalue of $G_{11}$ at distance $\Delta$ from
the other eigenvalues, on a circle of radius $r<\Delta$ about it the length contributes
$\abs\Gamma/2\pi=r$, while $1/\sig_{11}$, and with it $M$ and $M_0$, grows like $1/r$ as $r$ decreases
($1/\sig_{22}$ does not grow); the product then behaves like $\max(\gamma_1,\gamma_2)/r$ and is
smallest at $r$ about $\Delta/2$. A Jordan block of size $\ell$ makes $1/\sig_{11}$ grow like
$r^{-\ell}$ and the product like $r^{1-2\ell}$. What is lost on a small contour is the
cancellation in the integral, which bounding the integrand pointwise does not see.
\end{remark}

\subsection{Self-adjoint operators}\label{sec:selfadj}

\begin{corollary}\label{cor:selfadj}
Let $A=A^{*}$, let $\theta_1\ge\dots\ge\theta_k$ be the eigenvalues of $G_{11}$, and let $\bar
g\ge\norm{G_{22}}$ and $\bar\gamma\ge\norm{P_\perp AQ}$. Then
\begin{enumerate}[label=\textup{(\roman*)},nosep]
\item every point $\lambda$ of the spectrum of $A$ with $\abs\lambda>\bar g$ satisfies
  $\min_i\abs{\lambda-\theta_i}(\abs\lambda-\bar g)\le\bar\gamma^{2}$;
\item if $\bar g<x_1<x_2$ and $x_1$, $x_2$ are at distance more than $\bar\gamma^{2}/(x_1-\bar g)$
  from every $\theta_i$, then $A$ and $G_{11}$ have the same number of eigenvalues in $[x_1,x_2]$,
  counted with multiplicity, and the same holds for $[-x_2,-x_1]$.
\end{enumerate}
\end{corollary}

\begin{proof}
Since $G_{21}=G_{12}^{*}$, both couplings have norm at most $\bar\gamma$, and since $G_{11}$ is
Hermitian, $\sig_{11}(z)=\min_i\abs{z-\theta_i}$; the Main Theorem gives (i). For (ii) let $\Gamma$
be the circle with diameter $[x_1,x_2]$ and $D$ the open disc it bounds. Every $z\in D\cup\Gamma$ has
$\abs z\ge\operatorname{Re}z\ge x_1>\bar g$, and for $z\in\Gamma$ and real $\theta$,
$\abs{z-\theta}$ is at least the distance from $\theta$ to the nearer of $x_1$, $x_2$; apply
Proposition~\ref{prop:count}. Both spectra are real, and neither endpoint is an eigenvalue of $A$, by
(i), or of $G_{11}$, by hypothesis. For $[-x_2,-x_1]$ apply the same to $-A$.
\end{proof}

The radius in (i) is quadratic in the coupling, of the same form as the bound of Kato and
Temple~\cite{Temple1928,Kato1949}, which needs an interval containing no other eigenvalue of $A$;
here the gap is measured to $\bar g$. In finite dimension the ordered pairing gives a radius linear in the
coupling, better where $\bar\gamma>\abs\lambda-\bar g$.

\begin{corollary}\label{cor:weyl}
In Corollary~\ref{cor:selfadj} let $\mathcal H=\C^n$, order the eigenvalues of $A$ as
$\lambda_1\ge\dots\ge\lambda_n$, and let $j_+=\#\{i:\theta_i>\bar g\}$. Then
$\abs{\lambda_i-\theta_i}\le\bar\gamma$ for $i\le j_+$ and $\lambda_{j_++1}\le\bar g+\bar\gamma$, and
symmetrically at the bottom of the spectrum.
\end{corollary}

\begin{proof}
In the basis $[\,Q\mid Q_\perp\,]$, $A$ is unitarily similar to $D+E$ with $D=G_{11}\oplus G_{22}$
and $E=\bigl(\begin{smallmatrix}0&G_{12}\\G_{12}^{*}&0\end{smallmatrix}\bigr)$; $E^{*}E$ is block
diagonal, so $\norm E\le\bar\gamma$. Then $D-\bar\gamma I\le D+E\le D+\bar\gamma I$ and the ordered
eigenvalues are monotone~\cite[I-Thm.~6.44]{Kato1995}, so
$\abs{\lambda_i-\lambda_i(D)}\le\bar\gamma$. The spectrum of $D$ is
$\{\theta_i\}\cup\spec(G_{22})$ with $\spec(G_{22})\subseteq[-\bar g,\bar g]$.
\end{proof}

\begin{remark}\label{rem:lanczos}
After $k$ steps of Lanczos, $AQ_k=Q_kT_k+\beta_kq_{k+1}e_k^{*}$, take for $Q$ the Ritz vectors
$Q_kS_{\mathcal C}$ whose residuals $r_i=\beta_k\abs{(s_i)_k}$ are below a tolerance. Then
$P_\perp AQ=\beta_kq_{k+1}e_k^{*}S_{\mathcal C}$ has rank one and norm
$(\sum_{i\in\mathcal C}r_i^{2})^{1/2}$ in exact arithmetic, so the radii above are at the convergence
tolerance, while with the whole Krylov basis as frame $\norm{P_\perp AQ}=\beta_k$, which need not
decay. In floating point $\bar\gamma$ is taken from the Gram bound of Section~\ref{sec:coupling} on
the frame actually computed, not from the $r_i$. Each Ritz value $\theta_i$ lies within $r_i$ of the
spectrum of $A$, by~\cite[Lemma~1]{DaviesPlum2004} applied to the frame of its Ritz vector alone;
what no residual gives is the count, which is the obstruction of Proposition~\ref{prop:floor}, and
$\bar g$ is what supplies it.
\end{remark}

\subsection{How far the bound reaches}\label{sec:reach}

The essential norm is $\norm A_{\mathrm{ess}}=\inf\{\norm{A-C}:C\ \text{compact}\}$, and it is
submultiplicative, since $ST$ and $(S-C_1)(T-C_2)$ differ by a compact operator; so
$\norm{A^m}_{\mathrm{ess}}$ is a submultiplicative sequence and
\begin{equation}\label{eq:ress}
  r_{\mathrm{ess}}(A)\;=\;\lim_{m\to\infty}\norm{A^m}_{\mathrm{ess}}^{1/m}
                     \;=\;\inf_{m\ge1}\norm{A^m}_{\mathrm{ess}}^{1/m}
\end{equation}
exists; it is the essential spectral radius. For $\abs z>r_{\mathrm{ess}}(A)$ the operator $A-zI$ is
Fredholm, and $\{\abs z>r_{\mathrm{ess}}(A)\}$ is connected and meets the resolvent set, so there the
spectrum of $A$ consists of isolated eigenvalues of finite algebraic
multiplicity~\cite[IV-\S5.6]{Kato1995}. On a Banach space $r_{\mathrm{ess}}$ is defined through a
measure of noncompactness, and \eqref{eq:ress} is a theorem~\cite{Nussbaum1970}.

\begin{proposition}\label{prop:essfloor}
For every finite-rank isometry $Q$ and every $\bar g\ge\norm{G_{22}}$,
\[
  \bar g\;\ge\;r_{\mathrm{ess}}(A).
\]
\end{proposition}

\begin{proof}
Since $A-G_{22}=QQ^{*}A+P_\perp AQQ^{*}$ has rank at most $2k$, the two operators differ by a compact
one and have the same essential norm. Hence $\norm{G_{22}}\ge\norm{A}_{\mathrm{ess}}\ge
r_{\mathrm{ess}}(A)$ by \eqref{eq:ress}.
\end{proof}

The region in which the method reads its answer is $\abs z>\bar g$, and what it returns there
is read off $G_{11}$: a finite set of eigenvalues with their multiplicities.
Proposition~\ref{prop:essfloor} places that threshold at or above $r_{\mathrm{ess}}(A)$ and says
nothing about how much above.

%=======================================================================
\section{The small blocks}\label{sec:small}
%=======================================================================

From here to Section~\ref{sec:split} the space is $\C^n$: $\sv_1\ge\sv_2\ge\dots$ are the singular
values, $\norm\cdot_1$, $\norm\cdot_\infty$ and $\norm\cdot_F$ carry subscripts, $u$ is the unit
roundoff, $e_j$ the canonical basis, and cost is counted in applications of $A$ or $A^{*}$ and in stored vectors of length
$n$.

\subsection{The singular value enclosure}\label{sec:dense}

\begin{theorem}[{\cite[Thm.~10]{Miyajima2014}}]\label{thm:svdbox}
Let $A\in\C^{n_1\times n_2}$, $q=\min(n_1,n_2)$, and let $\widehat U\in\C^{n_1\times q}$, $\widehat
V\in\C^{n_2\times q}$ and diagonal $\widehat\Sigma\in\R^{q\times q}$ with
$\widehat\Sigma_{11}\ge\dots\ge\widehat\Sigma_{qq}\ge0$ be computed approximations. Put
$\widehat F=\widehat V^{*}\widehat V-I_q$, $\widehat G=\widehat U^{*}\widehat U-I_q$ and, when
$n_1\ge n_2$,
\[
  \rho=\frac{\norm{A\widehat V-\widehat U\widehat\Sigma}}{\sqrt{1-\norm{\widehat F}}},\qquad
  \underline c=\sqrt{\frac{1-\norm{\widehat G}}{1+\norm{\widehat F}}},\qquad
  \overline c=\sqrt{\frac{1+\norm{\widehat G}}{1-\norm{\widehat F}}},
\]
and when $n_1<n_2$ the same with the residual $\widehat U^{*}A-\widehat\Sigma\widehat V^{*}$ and with
$\widehat F$ and $\widehat G$ exchanged. If $\norm{\widehat F}<1$ and $\norm{\widehat G}<1$, then for
$i=1,\dots,q$
\[
  \underline c\,\widehat\Sigma_{ii}-\rho\;\le\;\sv_i(A)\;\le\;\overline c\,\widehat\Sigma_{ii}+\rho .
\]
\end{theorem}

Miyajima states this for real matrices; the proof uses only the singular value inequalities for a
product and for a perturbation, so it holds over $\C$ with the transposes read as adjoints. The order
of $\widehat\Sigma$ is the one a singular value decomposition routine returns, and the proof uses it
through $\sv_i(\widehat\Sigma)=\widehat\Sigma_{ii}$.

Every matrix we enclose is a ball matrix $\mathbf B$ with midpoint $M$ and entrywise radius $R\ge0$.
We apply Theorem~\ref{thm:svdbox} to $M$, with the residual and the two defects formed in ball
arithmetic, and widen each interval by an upper bound on $\norm R$: for every $X$ with
$\abs{X-M}\le R$ entrywise,
\[
  \abs{\sv_i(X)-\sv_i(M)}\;\le\;\norm{X-M}\;\le\;\norm R,
\]
the first inequality being Weyl's: with $\mathcal D(Y)$ the Hermitian matrix
$\bigl(\begin{smallmatrix}0&Y\\Y^{*}&0\end{smallmatrix}\bigr)$, whose nonzero eigenvalues are the
$\pm\sv_i(Y)$ with $\sv_i(Y)>0$, the inequalities
$\mathcal D(M)-\norm{X-M}I\preceq\mathcal D(X)\preceq\mathcal D(M)+\norm{X-M}I$ and the monotonicity
of ordered eigenvalues~\cite[I-Thm.~6.44]{Kato1995} give it. The second is the monotonicity of the
spectral norm in the moduli of the entries~\cite[Lemma~6.6]{Higham2002}. Both index the singular values in decreasing
order, as the theorem does. We write $\svdbox(\mathbf B)$ for the interval vector this returns and
$\svdbox(\mathbf B)_1$ for the enclosure of $\norm{\mathbf B}$. The residual norms are bounded by any
immediate bound, the Frobenius norm among them: they are small, and a crude bound on a small quantity
is enough, which is what makes the apparent circularity harmless~\cite[p.~372]{Rump2011}.

The enclosure is two-sided, so the call that bounds $\norm{\mathbf B}$ from above bounds
$\sigmin(\mathbf B)$ from below, which is what the circle certificates of Section~\ref{sec:split} use.
Handing the ball $\mathbf B$ itself to Theorem~\ref{thm:svdbox} would bring $R$ into the residual
through the product $\mathbf B\widehat V$, whose radius is bounded entrywise by $R\abs{\widehat V}$,
and $\norm{R\abs{\widehat V}}$ can exceed $\norm R\norm{\widehat V}$ by a factor up to $\sqrt q$; the
widening keeps $R$ out of the theorem. Miyajima's Theorem~7 encloses $\sv_i(A)$ instead by
$\widehat\Sigma_{ii}\sqrt{(1\pm\norm{\widehat F})(1\pm\norm{\widehat G})}\pm\norm{\widehat E}$, with
$\widehat E=\widehat U\widehat\Sigma\widehat V^{*}-A$; its bounds are never worse than those from
square frames~\cite[Thm.~8]{Miyajima2014}, and in Miyajima's measurements the two theorems give
comparable bounds~\cite[Table~2]{Miyajima2014}, Theorem~\ref{thm:svdbox} at a cost of
$22\max(n_1,n_2)q^{2}+12q^{3}$ operations against $24\max(n_1,n_2)q^{2}+12q^{3}$.
Oishi~\cite{Oishi2001} bounds the singular values from an approximate singular value decomposition
with rounding-mode controlled arithmetic; Rump~\cite{Rump2011} gives four methods spanning a factor
of eight in cost; Rump and Lange~\cite{RumpLange2023} extend the line to all singular pairs with
individual bounds that hold through clusters.

\begin{remark}\label{rem:rumptrunc}
Rump's two-frame bound~\cite[Thm.~3.1]{Rump2011} runs in the other direction, mapping $A$ into the
frame coordinates as $U^{*}AV=D+E$ and reading the singular values off $D$, so the residual is the
off-diagonal part of the transformed matrix rather than a comparison with $A$. Truncating $U$ to its
first $q$ columns deletes the rows of $E$ the discarded columns carry, and while the lower bound
still holds, the upper bound can fail, on a $2n\times n$ example whose truncated residuals all vanish
although every singular value equals one~\cite[Rem.~3]{Miyajima2014}.
\end{remark}

\subsection{The compression}\label{sec:vbd}

We bound $\sig_{11}(z)=\sigmin(G_{11}-zI)$. Inside $G_{11}$ three groups of eigenvalues appear:
those outside the circle of radius $R$, which are to be enclosed and counted; those deep inside,
which affect nothing outside and need only be kept from leaving the disc; and those inside whose
pseudospectrum reaches outside, which have to be carried at full accuracy although they are never
reported. The first and the last bear on the statements outside the circle.

Section~\ref{sec:dense} encloses $\sigmin(G_{11}-zI)$ by one $\svdbox(G_{11}-zI)$ at $O(k^{3})$ per
sample, $G_{11}$ being a ball matrix. The eigenvalues of $G_{11}$ enter the certificate only through their number in a region,
which Proposition~\ref{prop:count} transfers to $A$. That number is read off inclusions of all the
eigenvalues of the computed compression by Theorem~2.2 of~\cite{Rump2022a}: when every cluster of
approximate eigenvalues passes the inclusion test of that theorem and the spectral radius condition
on the same iterate holds, checked through the Perron--Frobenius bound of its Remark~2.4, each cluster
of $m$ approximate eigenvalues carries a block of size $m$ whose eigenvalues lie in a disc about the
cluster, and the blocks together carry the whole spectrum. Where some cluster does not pass, the
theorem is applied again in a Schur frame, with the recursion of step~6 of Rump's algorithm on the
clusters that failed, a variant implemented in BallArithmetic.jl~\cite{BallArithmetic2023}; and
where that fails, the Gershgorin discs of a matrix similar to the computed compression serve
instead, the similarity being exact because the transformed matrix is enclosed through a verified
linear solve: every eigenvalue lies in the union of the discs, and a union of discs disjoint from the
others holds as many eigenvalues as it has discs~\cite[Thm.~1.6]{Varga2004}. In either case the count in a region is the sum of the counts of
the discs inside it, provided none meets its boundary, and Lemma~\ref{lem:homotopy} carries it to
$G_{11}$; where the eigenvalues outside the region lie is not needed. Elsewhere the
eigenvalues are used in floating point, as the centres of discs.

\subsubsection{Transport, and why the block route is not used}\label{sec:transport}

\begin{remark}\label{rem:vbdcond}
We tried the block-diagonalising route in several forms, each replacing the factorisation per $z$ by
an evaluation at $O(k)$: a verified block diagonalisation over an eigenvector frame, and over a Schur
frame with between-cluster Newton steps. On a circle at $1.2\norm{G_{11}}$ the off-block part of
the transformed matrix came to between $10^{-13}$ and $10^{-4}$, and its diagonal part alone came
within a factor of two and a half of $\sigmin(G_{11}-zI)$. The transport to the frame divides the
lower bound by the condition number $\kappa$ of the frame, and on a Grcar matrix $\kappa$ was $135$
at $k=16$ and $8.2\cdot10^{10}$ at $k=64$, taking the certified lower bound to $3\cdot10^{-11}$ of the
value the singular value decomposition returns. For a normal compression $\kappa=1$ and the same
evaluation costs $O(k)$ per sample against $O(k^{3})$.
\end{remark}

\subsection{The couplings}\label{sec:coupling}

The two coupling blocks are $n\times\dim\Qs$ up to an isometry, so each has a small Gram matrix:
\begin{equation}\label{eq:gram}
  \gamma_2^{2}=\norm{C^{*}C},\quad C=P_\perp AQ,
  \qquad
  \gamma_1^{2}=\norm{D^{*}D},\quad D=P_\perp A^{*}Q,
\end{equation}
Hermitian positive semidefinite of size $\dim\Qs$, formed in ball arithmetic from $2\dim\Qs$ applications of $A$ and $A^{*}$
and enclosed by $\svdbox$ at $O((\dim\Qs)^{3})$, once per operator and not per $z$. The products are
formed as $\tilde C=(I-QQ^{*})AQ$ and $\tilde D=(I-QQ^{*})A^{*}Q$ also when $Q$ is not orthonormal:
with $\Pi_\perp$ the orthogonal projection onto $\Qs^{\perp}$ one has $\Pi_\perp(I-QQ^{*})=\Pi_\perp$ and
$\norm{\Pi_\perp}\le1$, so $\norm{\Pi_\perp AQ}\le\norm{\tilde C}$ and
$\norm{Q^{*}A\Pi_\perp}\le\norm{\tilde D}$; for an isometry $\tilde C=C$ and $\tilde D=D$. For an operator applied to
vectors, the rounding of $AQ$ enters through $\varepsilon_A\norm A$, with $\norm A$ not known in
advance; the bounds are then affine in an upper bound on $\norm A$, certified with $\bar g$ in
Section~\ref{sec:fpsweep}. Theorem~\ref{thm:svdbox} applied to $C$ and $D$ directly is the
alternative, at comparable cost and with all $\dim\Qs$ singular values enclosed two-sided rather
than the largest alone, forming $C^{*}C$ squaring the conditioning so that the small singular values
are lost; we use it when they are wanted, and \eqref{eq:gram} otherwise. The sweep of Section~\ref{sec:matfree} also
applies to these two operators, with the overshoot of Proposition~\ref{prop:ladder} where the Gram
route is exact; it is the fallback for a $Q$ applied implicitly, whose columns are not available.

%=======================================================================
\section{The far block}\label{sec:norm}
%=======================================================================

The far block $G_{22}$ is $n\times n$, has no small representation and is reached only through
applications of $A$ and $A^{*}$, and it is on this block that the lower bound of
Proposition~\ref{prop:floor} applies.

\subsection{What the far block admits}\label{sec:lit}

We bound $B=G_{22}$ by the family $\norm B\le(\tr H^{m})^{1/2m}$ with
$H=B^{*}B$, whose first member is the Frobenius norm, together with the interpolation bound at every
order, which the same columns give at no further application (see \eqref{eq:mincert}). Every member is a function of all $n$ columns of
$B$, so each costs a full sweep, which by Proposition~\ref{prop:floor} is, up to the $k$
directions of the frame, what any deterministic bound costs. Rump collects the Frobenius norm at the head of~\cite[\S3]{Rump2011} with the
interpolation bound $\sqrt{\norm B_1\norm B_\infty}$ and the positive-test-vector bound of Collatz
applied to $\abs B^{\!\top}\abs B$: each can overestimate $\norm B$ by a factor up to $\sqrt n$, and
for the interpolation bound that factor is attained on Hadamard and Fourier matrices.

The other bounds available when the operator is reached only through applications, the
matrix-vector query model of Section~\ref{sec:matfree}, hold only with a user-chosen probability: statistics of a
single application~\cite{Dixon1983,HMT2011}, Lanczos bidiagonalisation
polynomials~\cite{Hochstenbach2013}, rank-one test vectors~\cite{BujanovicKressner2021} and their
refinements~\cite{Naumov2025}, and by
Proposition~\ref{prop:floor} none can be promoted to a certificate below $n-k$ applications. From a
deterministically chosen start vector, Krylov information does not solve the relative-error problem
for $\lambda_{\max}$ of a symmetric positive definite matrix at any number of steps, while a random
start makes it solvable on average~\cite{KW1992}. We use an estimator of this kind only to choose the
order and the power-of-two scale of the sweep (Remark~\ref{rem:impl}); nothing certified depends on
it.

The inequality of Section~\ref{sec:matfree} is classical and has been
made into an algorithm, Gram iteration~\cite{Delattre2023}: squaring $G^{(t+1)}=G^{(t)*}G^{(t)}$
with a rescaling against overflow gives $\sv_1(G)\le\norm{G^{(t)}}_F^{2^{1-t}}$, the order doubling
at each step. It has been made sound in exact rational arithmetic~\cite{Tobler2025} and in binary64
with a machine-checked error model~\cite{Murray2026}, and extended to any consistent
norm~\cite{Delattre2024}. All of these store the matrix explicitly, at $O(n^{3})$ flops and $O(n^{2})$
memory per step, and~\cite{Delattre2023} states that the method needs the matrix itself; that is the gap
Section~\ref{sec:matfree} fills, at any index rather than at powers of two, and at the price of
orders that grow one at a time.

The Collatz bound is cheap only for an entrywise nonnegative operator, and the far
block is not one: on $\abs{G_{22}}$ it needs every column of $G_{22}$, at $n$ applications of the
operator, and overestimates (Remark~\ref{rem:positive}).

\subsection{Bounding the far block from applications of the operator}\label{sec:matfree}

In this subsection $B$ is a linear operator on $\C^n$ known only through $x\mapsto Bx$ and $x\mapsto
B^{*}x$. An upper bound on $\norm{B}$ is the awkward direction for a projection method: for
$H=H^{*}\succeq0$ and any subspace $\mathcal S$, the largest Ritz value is
$\theta_{\max}=\max_{0\ne x\in\mathcal S}\langle Hx,x\rangle/\langle x,x\rangle\le\lambda_{\max}(H)$,
which is the wrong side, and the deficit is the angle between $\mathcal S$ and the maximising
eigenvector, which no residual computed inside $\mathcal S$ sees. This is a property of compressions,
and a trace is not a compression. The model is the matrix-vector query model, in which the number of
queries is the complexity measure and lower bounds are proved against adaptive
algorithms~\cite{SunEtAl2021}.

\begin{proposition}\label{prop:floor}
Let $B$ be a linear operator on $\C^n$, and let an algorithm make $N<n$ queries $x_1,\dots,x_N$ to
$B$ or $B^{*}$, each chosen adaptively from the previous answers, and return a number claimed to
bound $\norm{B}$ from above. Then no finite number it returns is an upper bound for every operator
consistent with the answers received.
\end{proposition}

\begin{proof}
Since $N<n$, choose a unit $v$ orthogonal to $x_1,\dots,x_N$ and put $B_\mu=B+\mu vv^{*}$. Then
$B_\mu x_i=Bx_i$ and $B_\mu^{*}x_i=B^{*}x_i$ for every $i$, so by induction on the queries the
algorithm receives the same answers on $B$ and on $B_\mu$ and returns the same number, while
$\norm{B_\mu}\ge\mu-\norm B\to\infty$.
\end{proof}

When $B=P_\perp AP_\perp$ is the far block of a known frame, the same argument with $v\in\ran
P_\perp$ orthogonal to the queries shows that $n-k$ queries are necessary, and $n-k$ applications
along a basis of $\ran P_\perp$ determine $B$. The argument applies to any finite number of queries
on an infinite-dimensional space.

\begin{remark}\label{rem:positive}
Known structure can break Proposition~\ref{prop:floor}: for $B\ge0$ entrywise,
$\norm B^2\le\max_i(B^{*}Bx)_i/x_i$ for every $x>0$, the bound of Collatz, at two applications. It does
not help here, the far block being small through cancellation and not nonnegative, so the bound has
to be taken on $\abs{G_{22}}$, which needs every column of $G_{22}$, at $n$ applications of the
operator, and it never falls below
$\norm{\,\abs{G_{22}}\,}$.
\end{remark}

\begin{proposition}\label{prop:gelfand}
Let $\tnorm{\cdot}$ be any submultiplicative matrix norm and $H=B^{*}B$. Then $\norm{B}\le\tnorm{H^{m}}^{1/2m}$
for every $m\ge1$, and $\tnorm{H^{m}}^{1/2m}\to\norm{B}$.
\end{proposition}

\begin{proof}
As $H\succeq0$, $\rho(H)=\lambda_{\max}(H)=\norm{B}^{2}$. All matrix norms are equivalent, so
$\lim_m\tnorm{X^{m}}^{1/m}=\rho(X)$~\cite[III-(6.13)]{Kato1995}, and submultiplicativity gives
$\tnorm{X^{jm}}^{1/jm}\le\tnorm{X^{m}}^{1/m}$, so $\rho(X)\le\tnorm{X^{m}}^{1/m}$ for every $m$ with
equality in the limit. Take $X=H$ and square roots.
\end{proof}

Applied to $B$ itself the same limit bounds the spectral radius, and for non-normal $B$ the sequence
$\tnorm{B^{k}}^{1/k}$ converges to $\rho(B)\le\norm{B}$, so the inequality fails; the passage through
$H=B^{*}B$ is what avoids that.

\begin{proposition}\label{prop:ladder}
Let $B\ne0$, $H=B^{*}B$, $r=\rank B$, and $\phi_m:=\tr H^{m}/\sv_1^{2m}\in[1,r]$. Then for every integer
$m\ge1$
\begin{equation}\label{eq:ladder}
  \norm{B}\;\le\;\big(\tr H^{m}\big)^{1/2m}\;=\;\phi_m^{1/2m}\,\norm{B}
  \;\le\;r^{1/2m}\,\norm{B} ,
\end{equation}
and the middle term is nonincreasing in $m$, converging to $\norm{B}$.
\end{proposition}

\begin{proof}
Since $\tr H^m=\sum_i\sv_i^{2m}$, the middle term is $\norm{B}_{S_{2m}}$ and \eqref{eq:ladder} reads
$\sv_1^{2m}\le\sum_i\sv_i^{2m}\le r\sv_1^{2m}$; monotonicity in $m$ is that of the $\ell^p$ norms of
$(\sv_i)$.
\end{proof}

Each $m$ is a certificate on its own, so the algorithm may raise $m$ adaptively and stop at the first
order clearing its target; the factor $r^{1/2m}$ enters only when predicting how large $m$ must be,
where $n$ is a valid substitute for $r$. The order $m=1$ uses $n$ applications, against the $n-k$
that Proposition~\ref{prop:floor} shows to be necessary for the far block, the sweep running over the
canonical basis rather than a basis of $\ran P_\perp$; the higher orders spend further applications
on accuracy and not on certification.

For $m=2p$ even, $H^{p}$ is Hermitian and
\begin{equation}\label{eq:sweep}
  \tr H^{2p}\;=\;\norm{H^{p}}_F^{2}\;=\;\sum_{j=1}^{n}\norm{H^{p}e_j}^{2},
\end{equation}
and for $m=2p+1$ odd, $\tr H^{2p+1}=\sum_j\norm{BH^{p}e_j}^{2}$, one further application per column.
Every order is a sum of squared norms at $m$ applications of $B$ or $B^{*}$ per column. Writing it
this way rather than as $\sum_j\langle e_j,H^{m}e_j\rangle$, which is the same number, halves the
applications and replaces a signed sum by a sum of nonnegative terms, each bounded from above in
Section~\ref{sec:fpsweep}. If $B=P_\perp AP_\perp$ then $H$ vanishes on $(\ran P_\perp)^{\perp}$ and
the trace over $\C^n$ equals the trace over $\ran P_\perp$, so the sweep runs over the canonical basis
and no basis of $\ran P_\perp$ is constructed.

The first vector of every column chain is $Be_j$, and from these the
interpolation bound is accumulated at no further application: since $\norm
B^2=\rho(B^{*}B)\le\norm{B^{*}B}_1\le\norm B_\infty\norm B_1$,
\begin{equation}\label{eq:interp}
  \norm{B}\;\le\;\sqrt{\norm{B}_1\,\norm{B}_\infty},\qquad
  \norm{B}_1=\max_j\norm{Be_j}_1,\quad
  \norm{B}_\infty=\Big\|\sum_j\abs{Be_j}\Big\|_\infty ,
\end{equation}
and the same holds at every order of the chain. Let $M_t$ be the matrix whose $j$-th column is the
$t$-th vector of the chain of column $j$, so that $M_{2i}=H^{i}$ and $M_{2i+1}=BH^{i}$; then
$\norm{M_{2i}}=\norm B^{2i}$ and $M_{2i+1}^{*}M_{2i+1}=H^{2i+1}$, so $\norm{M_t}=\norm B^{t}$ for every
$t$, and \eqref{eq:interp} applied to $M_t$ gives
\begin{equation}\label{eq:interpt}
  \norm B\;\le\;\bigl(\norm{M_t}_1\,\norm{M_t}_\infty\bigr)^{1/2t},
\end{equation}
again at no further application, since the column sums and the accumulated row sums of the vectors of
order $t$ are formed as the vectors are computed. The certificate is
\begin{equation}\label{eq:mincert}
  \norm{B}\;\le\;\min_{1\le t\le m}\min\Big\{\bigl(\norm{M_t}_1\,\norm{M_t}_\infty\bigr)^{1/2t},\
  \big(\tr H^{t}\big)^{1/2t}\Big\},
\end{equation}
at an order $m$ fixed in advance, $8$ in Section~\ref{sec:numerics}; the order needs no certificate,
since \eqref{eq:ladder} and \eqref{eq:interpt} hold for every order. At $t=1$ the interpolation bound
is exact on an entrywise nonnegative matrix whose row sums and column sums all equal one constant.

\subsection{The sweep in floating point}\label{sec:fpsweep}

From here $B$ is the far block $G_{22}=P_\perp AP_\perp$ of a frame $Q\in\C^{n\times k}$ and the
order $m=2p$ is even; the odd orders need one more application of $B$ and the same argument. The
sweep costs $mn$ applications of the operator in $O(mn)$ memory, three vectors at a time and, for the
interpolation bound, one accumulated vector of row sums per order, the run for order $m$ returning
every lower order with it; the projector adds the $nk$ numbers of the
frame and $O(mn^{2}k)$ flops, of the same order as the $O(mn^{2}w)$ flops of the applications when $k$ is
comparable to the stencil width $w$ of a sparse $A$, and the $n$ columns are independent tasks that
communicate only in the final sum. It is the only step whose cost is $n$ applications of the operator
or more, so it runs in floating point and what has to be certified is its rounding.

\begin{remark}\label{rem:ballfails}
Ball arithmetic encloses $By$ through $\abs{P_\perp}\abs A\abs{P_\perp}$, which does not see the
cancellation that makes the far block small, so relative to the vectors the radii can grow by a
factor up to about $\norm A/\norm B$ at each application, and the smaller $\norm B$ is against
$\norm A$, the wider the enclosure. We propagate no enclosure at all.
\end{remark}

The one thing assumed about the operator is a constant $\varepsilon_A$ with
\begin{equation}\label{eq:roundA}
  \norm{\fl(Ay)-Ay}\le\varepsilon_A\norm A\norm y,\qquad
  \norm{\fl(A^{*}y)-A^{*}y}\le\varepsilon_A\norm A\norm y
\end{equation}
for every $y$; the value of $\norm A$ is not assumed, and is certified together with $\norm B$. For a
stored matrix with at most $w$ nonzeros in each row and column,
$\abs{\fl(Ay)-Ay}\le u_w\abs A\abs y$ componentwise with $u_w=wu/(1-wu)$, by the bound for inner
products of length at most $w$~\cite[(3.11)--(3.12)]{Higham2002}, where it is stated for dense products,
and $\norm{\,\abs A\,}\le\sqrt w\,\norm A$, so $\varepsilon_A=\sqrt w\,u_w$. A bound of this kind
holds when $Ay$ is formed as a sum of its terms and fails when it is formed as a difference of much
larger operators; the far block is of the second kind, which is why the argument closes on $\norm A$
and not on $\norm B$.

We apply the projector as $y-Q\fl(Q^{*}y)$, with no solve. With $\alpha\ge\norm{I-Q^{*}Q}$
as in Section~\ref{sec:almostortho}, writing $Q=Z(Q^{*}Q)^{1/2}$ with $Z^{*}Z=I$ gives
$QQ^{*}-ZZ^{*}=Z(Q^{*}Q-I)Z^{*}$, so $QQ^{*}$ is at distance at most $\alpha$ from $I-P_\perp$. With
$\norm Q^{2}\le1+\alpha$, $\norm{\,\abs Q\,}\le\sqrt k\,\norm Q$~\cite[Lem.~6.6]{Higham2002} and the
componentwise bound for inner products of lengths $n$ and $k$,
\begin{equation}\label{eq:epsP}
  \norm{\fl\bigl(y-Q\fl(Q^{*}y)\bigr)-P_\perp y}\le\varepsilon_P\norm y,\qquad
  \varepsilon_P=\alpha+u+(1+u)(1+\alpha)\sqrt k\,\bigl(u_n+u_k(1+\sqrt k\,u_n)\bigr),
\end{equation}
so the defect of the frame enters additively and the rounding as a term bounded by a constant times
$(n+k)\sqrt k\,u$.

These bounds assume no underflow. With gradual underflow each product carries a further absolute
error of at most half of $\eta$, the smallest positive subnormal number, so each rounding constant
acquires an absolute term: \eqref{eq:roundA} becomes
$\norm{\fl(Ay)-Ay}\le\varepsilon_A\norm A\norm y+\eta_A$, with $\eta_A$ a multiple of $\eta$ for a
stored matrix, and \eqref{eq:epsP}, Lemma~\ref{lem:onestep} and the bounds on norms below acquire
the corresponding terms. Their contributions to $\delta_{j,t}$ have nonnegative coefficients and
degree at most $t$ in $(a,g)$, so the homogeneity in the proof of Proposition~\ref{prop:fpchecks}
still holds, and we omit them from the displayed formulas other than \eqref{eq:normbound}.
Overflow is excluded by refusing every result in which an intermediate is not finite. For complex
data every $u_\ell=\ell u/(1-\ell u)$ is read as $u_{\ell+2}$, a complex product having relative error
at most $\sqrt2\,u_2\le u_3$ and a complex sum at most $u$~\cite[Lem.~3.5]{Higham2002}, where
$u_\ell$ is written $\gamma_\ell$. And the norms of the computed vectors are themselves bounded from
above in floating point: with $\hat s$ the sum $\sum_i\abs{x_i}^2$ computed in rounding to nearest,
\begin{equation}\label{eq:normbound}
  \norm x^{2}\;\le\;\frac{\hat s}{1-u_n}+2n\eta ,
\end{equation}
the division by $1-u_n$, the addition and the square root taken for $\norm x$ being rounded upward.

\begin{lemma}\label{lem:onestep}
Let $\fl(By)=\fl(P_\perp\fl(A\fl(P_\perp y)))$ with each projection computed as in~\eqref{eq:epsP},
and let $\norm A\le a$, $\norm{G_{12}}\le\gamma_1$, $\norm{G_{21}}\le\gamma_2$ and
$\norm{G_{22}}\le g$. Then for every $y$
\[
  \norm{\fl(By)-By}\le \varepsilon_B\norm y,\qquad
\varepsilon_B=\varepsilon_P\sqrt{\gamma_2^2+g^2}+\varepsilon_Aa(1+\varepsilon_P)
 +\varepsilon_P\Bigl(\sqrt{\gamma_1^2+g^2}+a\bigl(\varepsilon_P+\varepsilon_A(1+\varepsilon_P)\bigr)\Bigr),
\]
and the same holds for $B^{*}$ with $\gamma_1$ and $\gamma_2$ exchanged.
\end{lemma}

\begin{proof}
Write $y_1=\fl(P_\perp y)=P_\perp y+f_1$, $y_2=\fl(Ay_1)=Ay_1+f_2$ and $y_3=\fl(P_\perp
y_2)=P_\perp y_2+f_3$, so $y_3-By=P_\perp Af_1+P_\perp f_2+f_3$. With $\Pi=I-P_\perp$, the operator
$P_\perp A=P_\perp A\Pi+P_\perp AP_\perp$ acts on the orthogonal summands $\ran\Pi$ and $\ran
P_\perp$, so $\norm{P_\perp A}\le\sqrt{\gamma_2^2+g^2}$ by Cauchy--Schwarz; the two terms of
$AP_\perp$ have orthogonal ranges, so $\norm{AP_\perp}\le\sqrt{\gamma_1^2+g^2}$. Hence
$\norm{P_\perp Af_1}\le\varepsilon_P\sqrt{\gamma_2^2+g^2}\norm y$,
$\norm{f_2}\le\varepsilon_Aa(1+\varepsilon_P)\norm y$, and $\norm{f_3}\le\varepsilon_P\norm{y_2}$
with $\norm{y_2}\le\bigl(\sqrt{\gamma_1^2+g^2}+a\varepsilon_P+\varepsilon_Aa(1+\varepsilon_P)\bigr)\norm y$.
\end{proof}

Only the application of $A$ carries $\norm A$; the errors of the projector are measured by the couplings and the
far block. For column $j$ let $w_0=e_j$, $v_i=\fl(Bw_i)$, $w_{i+1}=\fl(B^{*}v_i)$, let
$\varepsilon_B$ be the larger of the two constants of Lemma~\ref{lem:onestep}, and $g\ge\norm B$.
Then $\norm{w_{i+1}-Hw_i}\le\varepsilon_B(g\norm{w_i}+\norm{v_i})$, and by induction
\[
  \norm{H^{p}e_j-w_p}\le \varepsilon_B\,E_j(g),\qquad
  E_j(g)=\sum_{i=0}^{p-1}g^{2(p-1-i)}\bigl(g\norm{w_i}+\norm{v_i}\bigr).
\]
Let $\bar s(a)\ge\norm{G_{11}}$, $\bar\gamma_1(a)\ge\norm{G_{12}}$ and
$\bar\gamma_2(a)\ge\norm{G_{21}}$ be the bounds computed from $\fl(AQ)$ and $\fl(A^{*}Q)$ as in
Section~\ref{sec:coupling}, valid whenever $\norm A\le a$; each is a certified number plus $a$ times
a rounding constant, and for a frame that is not orthonormal $\bar s$ carries the factor
$1/(1-\alpha)$ of the proof of Lemma~\ref{lem:homotopy} and $\bar\gamma_1$, $\bar\gamma_2$ the factor
$1/\sqrt{1-\alpha}$ of Proposition~\ref{prop:almostortho}. With $\varepsilon_B(a,g)$ the constant
at $\bar\gamma_i(a)$, $c_j\ge\norm{w_p}$, $N_1\ge\max_j\norm{v_0}_1$ and
$N_\infty\ge\norm{\sum_j\abs{v_0}}_\infty$, \eqref{eq:ladder}, \eqref{eq:sweep} and
\eqref{eq:interp} give, for $a\ge\norm A$ and $g\ge\norm B$,
\begin{equation}\label{eq:FI}
  \begin{aligned}
  \norm B^{4p}&\le F_a(g):=\sum_j\bigl(c_j+\varepsilon_B(a,g)E_j(g)\bigr)^{2},\\
  \norm B^{2}&\le I_a(g):=\bigl(N_1+\sqrt n\,\varepsilon_B(a,g)\bigr)\bigl(N_\infty+n\,\varepsilon_B(a,g)\bigr).
  \end{aligned}
\end{equation}
At a higher order the rounding of the whole chain enters. Let $x_{j,0}=e_j$ and let $x_{j,t}$ be the
$t$-th computed vector of the chain of column $j$, so that $x_{j,2i+1}=v_i$ and $x_{j,2i+2}=w_{i+1}$,
and let $\nu_{j,t}\ge\norm{x_{j,t}}$. Lemma~\ref{lem:onestep} at each application and $\norm B\le g$
give $\norm{x_{j,t}-M_te_j}\le\varepsilon_B\nu_{j,t-1}+g\norm{x_{j,t-1}-M_{t-1}e_j}$, hence
\[
  \norm{x_{j,t}-M_te_j}\le\delta_{j,t}(a,g):=\varepsilon_B(a,g)\sum_{s=0}^{t-1}g^{t-1-s}\nu_{j,s}.
\]
Since $\norm{M_t}_F^{2}=\tr H^{t}$, \eqref{eq:ladder} gives at every order
\begin{equation}\label{eq:Ft}
  \norm B^{2t}\le F_{a,t}(g):=\sum_j\bigl(\nu_{j,t}+\delta_{j,t}(a,g)\bigr)^{2},
\end{equation}
and at $t=2p$ it bounds $F_a(g)$ from above, with $c_j=\nu_{j,2p}$ and the norms in $E_j$ replaced by
their bounds $\nu_{j,s}$.
Since the $\ell^1$ norm of a vector is at most $\sqrt n$ times its Euclidean norm and the
$\ell^\infty$ norm at most its Euclidean norm, \eqref{eq:interpt} gives, with
$N_{1,t}\ge\max_j\norm{x_{j,t}}_1$ and $N_{\infty,t}\ge\norm{\sum_j\abs{x_{j,t}}}_\infty$,
\begin{equation}\label{eq:FIt}
  \norm B^{2t}\le I_{a,t}(g):=\Bigl(N_{1,t}+\sqrt n\max_j\delta_{j,t}(a,g)\Bigr)
  \Bigl(N_{\infty,t}+\sum_j\delta_{j,t}(a,g)\Bigr),
\end{equation}
and at $t=1$, where $\delta_{j,1}=\varepsilon_B$, this is $I_a(g)$.

\begin{proposition}\label{prop:fpchecks}
Suppose that for a pair $(\bar a,\bar g)$ of positive numbers
\[
  \text{(i)}\ \ F_{\bar a,t}(\bar g)<\bar g^{2t}\ \text{ or }\ I_{\bar a,t}(\bar g)<\bar g^{2t}\
  \text{for some }t\le m,\qquad
  \text{(ii)}\ \ \left\|\begin{pmatrix}\bar s(\bar a)&\bar\gamma_1(\bar a)\\
  \bar\gamma_2(\bar a)&\bar g\end{pmatrix}\right\|<\bar a .
\]
Then $\norm A<\bar a$ and $\norm{G_{22}}<\bar g$.
\end{proposition}

\begin{proof}
For $\lambda\ge1$, $F_{\lambda a,t}(\lambda g)\le\lambda^{2t}F_{a,t}(g)$ and $I_{\lambda a,t}(\lambda
g)\le\lambda^{2t}I_{a,t}(g)$, since $\varepsilon_B(\lambda a,\lambda g)\le\lambda\varepsilon_B(a,g)$,
every term of $\delta_{j,t}$ has degree at most $t-1$ in $g$, and $\bar s$, $\bar\gamma_i$ are affine in $a$
with nonnegative coefficients, which also gives $M(\lambda a,\lambda g)\le\lambda M(a,g)$ entrywise
for the matrix $M$ of (ii). Let $\lambda=\max(\norm A/\bar a,\norm B/\bar g)$ and suppose
$\lambda\ge1$. If $\lambda=\norm A/\bar a$ then $\norm B\le\lambda\bar g$, and since $A$ is unitarily
similar to its block form and the norm is monotone in the moduli of the
entries~\cite[Lem.~6.6]{Higham2002}, $\norm A\le\norm{M(\lambda\bar a,\lambda\bar
g)}\le\lambda\norm{M(\bar a,\bar g)}<\lambda\bar a=\norm A$. Otherwise $\norm B=\lambda\bar g$ and
$\norm A\le\lambda\bar a$, and \eqref{eq:Ft} and \eqref{eq:FIt} at $(\lambda\bar a,\lambda\bar g)$
make one of the right-hand sides of (i) smaller than $\norm B^{2t}$. Both cases are
contradictions. The step from the blocks to $A$ uses
$\norm A\le\bigl\|\bigl[\norm{G_{ij}}\bigr]_{i,j}\bigr\|$, obtained by applying $A$ to $Qx_1+x_2$ and
bounding each block by its norm, before the monotonicity in the moduli of the entries.
\end{proof}

\begin{remark}\label{rem:impl}
The norm $\norm{H^{p}e_j}$ scales like $\norm{H}^{p}$ and overflows or underflows in binary64 as $m$
grows. The sweep therefore runs on $cA$, with $c$ the power of two that brings the least of the
floating-point estimates $v_t$ of $\norm{M_t}^{1/t}=\norm B$, $t\le m$, computed from eight Gaussian
vectors after eight warm-up applications, into $[1,2)$; scaling by a power of two is exact in binary
floating point apart from underflow, which the absolute terms cover, and the scale is unwound with
the final root. When the order is not fixed in advance it is chosen from the predicted bounds
$v_t(n-k)^{1/2t}$. The same device, with an explicit
runtime refusal when an intermediate is not finite, is what makes the dense iteration
of~\cite{Murray2026} executable.
\end{remark}

\begin{algorithm}[t]
\caption{The sweep of order $2p$ in floating point}\label{alg:fpsweep}
\begin{algorithmic}[1]
\Require $y\mapsto\fl(Ay)$ and $y\mapsto\fl(A^{*}y)$ with $\varepsilon_A$ as in~\eqref{eq:roundA};
  the frame $Q$ with $\alpha$ and $\varepsilon_P$ as in~\eqref{eq:epsP}; $\bar s(a)$,
  $\bar\gamma_1(a)$, $\bar\gamma_2(a)$ as functions of $a$
\For{$j=1,\dots,n$}
  \State $w_0\gets e_j$
  \For{$i=0,\dots,p-1$}
    \State $v_i\gets\fl(Bw_i)$, $w_{i+1}\gets\fl(B^{*}v_i)$ \Comment{rescaled as in Remark~\ref{rem:impl}}
  \EndFor
  \State add to running sums the products $\nu_{j,s}\nu_{j,s'}$ for $s,s'\le 2p$, and the $\nu_{j,s}$,
  rounded upward
  \State update $\max_j\nu_{j,s}$, $N_{1,s}$ and the vectors $\sum_j\abs{x_{j,s}}$ for $s\le 2p$,
  rounded upward
\EndFor
\State $\bar g\gets$ the least number passing (i) of Proposition~\ref{prop:fpchecks}, by bracketing
  and bisection, with $\bar a$ the least number passing (ii) at each $\bar g$ tried
\If{(i) and (ii) hold at $(\bar a,\bar g)$, the left-hand sides evaluated rounding upward and the
  powers of $\bar g$ rounding downward}
  \State\Return $\norm A<\bar a$ and $\norm{G_{22}}<\bar g$
\Else
  \State\Return failure
\EndIf
\end{algorithmic}
\end{algorithm}

Algorithm~\ref{alg:fpsweep} summarises the sweep. The running sums are $O(p^{2})$ scalars, with the
$2p$ accumulated vectors of row sums, from which every $F_{a,t}(g)$ and $I_{a,t}(g)$ is evaluated at
any pair after the sweep, so the bisections cost no further application.

%=======================================================================
\section{The certified splitting}\label{sec:split}
%=======================================================================

Let $A$ be large, fix $R>0$, and take the splitting \eqref{eq:introblocks} along a
frame $Q$ with $\Qs=\ran Q$.

\subsection{The frame}\label{sec:frame}

The kept subspace has two jobs needing different directions. The eigenvalues are read from a
compression, which wants eigendirections; the three blocks are singular-value quantities, and
Proposition~\ref{prop:frame} below bounds them through the leading singular vectors of $A$, which a
Krylov space of $A$ need not contain, and rotating inside a subspace cannot help since
$VS(VS)^{*}=VV^{*}$ for unitary $S$. What is needed are the singular vectors of $A$, produced by a
Golub--Kahan--Lanczos bidiagonalisation at two applications per step.

\begin{definition}\label{def:Q}
Fix $k_s$, let $V_\sig=[v_1\ \cdots\ v_{k_s}]$ and $U_\sig=[u_1\ \cdots\ u_{k_s}]$ be the leading right
and left singular vectors of $A$, and $W_\lambda$ the Arnoldi eigendirections with $\abs{\lambda}>R$.
Then $Q$ is the orthonormal factor of a rank-revealing factorisation of $[\,W_\lambda\mid
V_\sig\mid U_\sig\,]$, the numerically dependent columns being dropped.
\end{definition}

\begin{proposition}\label{prop:frame}
If $\Qs\supseteq\operatorname{span}V_\sig$ then $\norm{AP_\perp}\le\sv_{k_s+1}(A)$, so that
$\norm{G_{22}}\le\sv_{k_s+1}$ and $\gamma_1=\norm{Q^{*}AP_\perp}\le\sv_{k_s+1}$, while
$\gamma_2=\norm{P_\perp AQ}$ can be as large as $\sv_1$; if $\Qs\supseteq\operatorname{span}U_\sig$
then $\norm{A^{*}P_\perp}\le\sv_{k_s+1}(A)$, with the roles of $\gamma_1$ and $\gamma_2$ exchanged. When
both hold,
\[
\norm{G_{22}}\le\sv_{k_s+1},\qquad \gamma_1\le\sv_{k_s+1},\qquad \gamma_2\le\sv_{k_s+1},
\]
so that $\gamma_1\gamma_2\le\sv_{k_s+1}^{2}$, against $\sv_{k_s+1}\sv_1$ with one side only. Moreover
$\norm{G_{22}}\ge\sv_{2d+1}(A)$ for every frame of dimension $d$, so no frame of dimension $d$
brings the far block below $\sv_{2d+1}$.
\end{proposition}

\begin{proof}
For unit $x\in\ran P_\perp$ we have $x\perp v_1,\dots,v_{k_s}$, so $\norm{Ax}\le\sv_{k_s+1}$; the claim for
$U_\sig$ is the same for $A^{*}$. The bounds on $G_{22}$ and $\gamma_1$ follow from
$\norm{P_\perp}\le1$, and that on $\gamma_2$ from $\norm{G_{21}}=\norm{Q^{*}A^{*}P_\perp}$. With
$V_\sig$ alone nothing bounds $\gamma_2$ better than $\sv_1$, since the vector
$P_\perp Av_i=\sv_iP_\perp u_i$ has norm $\sv_i$ when $u_i\perp\Qs$. For the last claim,
$A-G_{22}=QQ^{*}A+AQQ^{*}-QQ^{*}AQQ^{*}$ has rank at most $2d$, and
$\sv_{i+j-1}(X+Y)\le\sv_i(X)+\sv_j(Y)$ with $i=1$, $j=2d+1$, $Y=A-G_{22}$ gives
$\sv_{2d+1}(A)\le\sv_1(G_{22})$.
\end{proof}

On an example, $A=U\Sigma V^{*}$ with random unitary $U$, $V$ of size $200$ and
$\sv_i=0.75^{i-1}$, at $R=0.15$ and $k_s=\#\{i:\sv_i\ge R\}=7$:

\begin{center}
\small
\begin{tabular}{lrccccc}
\toprule
& $\dim\Qs$ & frame & $\norm{G_{22}}$ & $\gamma_1$ & $\gamma_2$ & $\gamma_1\gamma_2$\\
\midrule
$\sv_8=0.134$ & $7$  & eigen $+$ right & $1.32\cdot10^{-1}$ & $2.92\cdot10^{-2}$ & $9.73\cdot10^{-1}$ & $2.84\cdot10^{-2}$\\
              & $7$  & eigen $+$ left  & $1.33\cdot10^{-1}$ & $9.92\cdot10^{-1}$ & $2.35\cdot10^{-2}$ & $2.33\cdot10^{-2}$\\
              & $14$ & both            & $1.31\cdot10^{-1}$ & $2.94\cdot10^{-2}$ & $2.39\cdot10^{-2}$ & $7.03\cdot10^{-4}$\\
\bottomrule
\end{tabular}
\end{center}

\noindent
With one side only the product is $2.84\cdot10^{-2}$ and $2.33\cdot10^{-2}$, against
$\sv_{k_s+1}\sv_1=0.134$ and $\sv_{k_s+1}^{2}=1.78\cdot10^{-2}$; with both it is $7.03\cdot10^{-4}$, lower
by factors of $40$ and $33$. Proposition~\ref{prop:frame} bounds the product by $\sv_{k_s+1}^{2}$ and no
more, so how far below that it falls has to be measured.

\subsection{The exterior}\label{sec:exterior}

Four certified numbers close the argument: $\bar g\ge\norm{G_{22}}$ from Section~\ref{sec:norm},
$\bar\gamma_1,\bar\gamma_2$ from \eqref{eq:gram}, and $\sig_{11}$ from Section~\ref{sec:vbd}.
Substituted into Theorem~\ref{thm:assembly} through \eqref{eq:farshort} they give the resolvent bound
at every $z$ with $\abs z>\bar g$ and $\sig_{11}(z)(\abs z-\bar g)>\bar\gamma_1\bar\gamma_2$, and
with it Propositions~\ref{prop:complete} and~\ref{prop:count} and Corollary~\ref{cor:agree}.

\begin{corollary}\label{cor:gap}
Let $R>\bar g$ and suppose the hypothesis of Proposition~\ref{prop:count} holds at every point of the
circle $\abs z=R$. If $G_{11}$ has exactly one eigenvalue of modulus $>R$, counted with algebraic
multiplicity, so does $A$, and every other eigenvalue of $A$ satisfies $\abs{\lambda}<R$.
\end{corollary}

\subsection{Choosing the frame, cost and failure}\label{sec:depth}

The frame of Definition~\ref{def:Q} has one free parameter at a given radius $R$, the number $k_s$ of
singular vectors per side, and it is chosen against the hypothesis of Theorem~\ref{thm:assembly} on
the circle $\abs z=R$. Starting from $k_s=0$ and then from $k_s=\#\{\sv_i(A)\ge R\}$, $k_s$ is
doubled until the certified $\bar g$ falls below $R$ and then until the cover of
Remark~\ref{rem:cover} completes, each time bisecting back between the last frame that failed and
the first that passed. Nothing makes the outcome monotone in $k_s$, so the frame returned is one at
which the search went from failing to passing, and a smaller one may pass; the search is uncertified,
and what is certified is the hypothesis on the frame it returns.

\begin{remark}\label{rem:twice}
If $\Qs_{k}\subseteq\Qs_{k+1}$ then
$P_{k+1}^{\perp}(P_{k}^{\perp}AP_{k}^{\perp})P_{k+1}^{\perp}=P_{k+1}^{\perp}AP_{k+1}^{\perp}$, so the
far block of a larger frame needs nothing from a smaller one, and the assembly applies unchanged to
whichever frame the search stops at.
\end{remark}

The two costs pull in opposite directions. The sweep is $mn$ applications once per frame; the verified
singular value decomposition of $G_{11}-zI$ is $O((\dim\Qs)^{3})$ per point of the cover, so a
larger frame costs cubically, multiplied by the number of points. This is the reason for taking the
least $k_s$ the search finds.

\begin{remark}\label{rem:limits}
By Proposition~\ref{prop:frame} no frame of dimension $d$ brings the far block below
$\sv_{2d+1}(A)$, so the singular-value decay of $A$ bounds from below the frame the search can
return (Section~\ref{sec:numnecessary}). The search runs with a cap on $\dim\Qs$, about $n/2$ in
Section~\ref{sec:numerics}, and declares failure when it reaches it.
\end{remark}

\subsection{The frame is not orthonormal}\label{sec:almostortho}

A computed frame satisfies $Q^{*}Q=I_k$ only to rounding, and the hypothesis that it be an isometry
cannot be certified. What can is a bound
\begin{equation}\label{eq:alpha}
  \alpha\;\ge\;\norm{I_k-Q^{*}Q},\qquad\alpha<1,
\end{equation}
the least of several certified bounds on the norm of the ball matrix $Q^{*}Q-I_k$, formed in ball
arithmetic, $\svdbox$ among them.

\begin{proposition}\label{prop:almostortho}
Let $Q\in\C^{n\times k}$ satisfy \eqref{eq:alpha} and put $\kappa^{2}=(1+\alpha)/(1-\alpha)$,
\[
  P_\perp=I-Q(Q^{*}Q)^{-1}Q^{*},\qquad
  \tilde\sig_{11}(z)=\sigmin\bigl(Q^{*}AQ-z\,Q^{*}Q\bigr),
\]
and let $\bar g\ge\norm{P_\perp AP_\perp}$, $\gamma_1\ge\norm{Q^{*}AP_\perp}$,
$\gamma_2\ge\norm{P_\perp AQ}$. Then at every $z$ with $\abs z>\bar g$ and
\begin{equation}\label{eq:alphacond}
  \tilde\sig_{11}(z)\,(\abs z-\bar g)\;>\;\kappa^{2}\,\gamma_1\gamma_2
\end{equation}
the conclusions of the Main Theorem hold for the splitting along the isometry $Z=Q(Q^{*}Q)^{-1/2}$,
whose range is $\ran Q$ and whose compression is $X=Z^{*}AZ$, with $\sig_{11}$ read as
$\tilde\sig_{11}/(1+\alpha)$ and the couplings as $\gamma_i/\sqrt{1-\alpha}$; in particular the
inclusion is into $\Lambda_\varepsilon(X)$ and the count is that of $X$.
\end{proposition}

\begin{proof}
From $(1-\alpha)I\preceq Q^{*}Q\preceq(1+\alpha)I$ the frame has full column rank, and its polar
decomposition is $Q=ZS$ with $Z^{*}Z=I_k$ and $S=(Q^{*}Q)^{1/2}$, so $\norm S\le\sqrt{1+\alpha}$ and
$\norm{S^{-1}}\le(1-\alpha)^{-1/2}$. Since $\ran Z=\ran Q$, the orthogonal projector onto the
complement is $P_\perp$ as displayed, so the far block and $\bar g$ are those of the isometry $Z$.
For the compression, $Q^{*}AQ-zQ^{*}Q=S(Z^{*}AZ-zI)S$ gives
$\sigmin(Z^{*}AZ-zI)\ge\tilde\sig_{11}/(1+\alpha)$. For the couplings,
$Z^{*}AP_\perp=S^{-1}Q^{*}AP_\perp$ and $P_\perp AZ=P_\perp AQ\,S^{-1}$, of norms at most
$\gamma_1/\sqrt{1-\alpha}$ and $\gamma_2/\sqrt{1-\alpha}$. Applying the Main Theorem to $Z$, its
hypothesis reads \eqref{eq:alphacond}.
\end{proof}

Four things change in the algorithm. The projector is applied as
$y-Q\fl(Q^{*}y)$ with the constant $\varepsilon_P$ of \eqref{eq:epsP}, which carries $\alpha$
additively. The compression is read as the pencil $Q^{*}AQ-zQ^{*}Q$, whose smallest singular value is
what Section~\ref{sec:vbd} certifies, at the same cost. The couplings are formed with $I-QQ^{*}$ as in
Section~\ref{sec:coupling}, which bounds those of $P_\perp$. And the threshold carries the single factor
$\kappa^{2}=(1+\alpha)/(1-\alpha)$ on the coupling product, which at the $\alpha\sim10^{-14}$ of a
computed orthonormalisation is $1+O(10^{-14})$ and at $\alpha=10^{-2}$ is $1.02$.

The counts of Proposition~\ref{prop:count} are about the eigenvalues of $Z^{*}AZ$, while what is
formed is $Q^{*}AQ$, as a ball matrix; the eigenvalues of its midpoint $\hat G$ are enclosed, and the number on each side of a
contour is carried over by the following lemma, in which the polar decomposition is used as
in~\cite{Nisoli2026}.

\begin{lemma}\label{lem:homotopy}
Let $Q$ satisfy \eqref{eq:alpha}, let $Q=ZS$ be its polar decomposition, and put $G=Q^{*}AQ$ and
$X=Z^{*}AZ$. Then
\begin{equation}\label{eq:etaalpha}
  \norm{G-X}\;\le\;\eta_\alpha:=\frac{\alpha\,\bigl(1+\sqrt{1+\alpha}\bigr)}{1-\alpha}\,\norm G .
\end{equation}
If $\hat G$ satisfies $\norm{\hat G-G}\le\rho$ and $\Gamma$ is a closed curve with
$\sigmin(X-zI)>\eta_\alpha+\rho$ at every $z\in\Gamma$, then $\hat G$ and $X$ have the same number of
eigenvalues, counted with algebraic multiplicity, in each connected component of $\C\setminus\Gamma$.
\end{lemma}

\begin{proof}
Since $G=SXS$, $G-X=(S-I)XS+X(S-I)$. The eigenvalues of $S$ are the square roots of those of $Q^{*}Q$,
which lie in $[1-\alpha,1+\alpha]$, so $\norm{S-I}\le1-\sqrt{1-\alpha}\le\alpha$,
$\norm S\le\sqrt{1+\alpha}$ and $\norm X\le\norm{S^{-1}}^{2}\norm G\le\norm G/(1-\alpha)$, which give
\eqref{eq:etaalpha}. Then $\norm{\hat G-X}\le\eta_\alpha+\rho$, and for $t\in[0,1]$ and $z\in\Gamma$,
$\sigmin(X+t(\hat G-X)-zI)\ge\sigmin(X-zI)-\norm{\hat G-X}>0$, so no matrix of the segment from $X$
to $\hat G$ has an eigenvalue on $\Gamma$; the eigenvalues depend continuously on $t$, and the number
in each component of $\C\setminus\Gamma$ is constant along the segment.
\end{proof}

In practice $G$ is a ball matrix and $\hat G$ its midpoint, with $\rho$ an upper bound on the spectral
norm of the matrix of radii: the eigenvalues are enclosed for the floating-point matrix $\hat G$, which
for an operator applied to vectors can be enclosed more tightly than the ball, whose radii carry
the rounding of $\fl(AQ)$ as a norm bound on every entry. The hypothesis costs nothing further: the
cover of Remark~\ref{rem:cover} already certifies a lower bound on $\sigmin(X-zI)$ along the whole
contour, and it is compared with $\eta_\alpha+\rho$.

\begin{remark}\label{rem:nokrylov}
No Krylov space appears in the certified path: the Arnoldi and Lanczos processes and the
rank-revealing factorisation enter only through $P_\perp$, and Proposition~\ref{prop:ladder} holds
for any $P_\perp$. The method declines when $\alpha<1$ is not certified; when the checks of
Proposition~\ref{prop:fpchecks} fail; when the hypothesis of Theorem~\ref{thm:assembly} is not
certified on the circle within the cap of Remark~\ref{rem:limits}; when the eigenvalue inclusions of
Section~\ref{sec:vbd} do not separate the circle; and when the lower bound on $\sigmin(X-zI)$ does
not exceed $\eta_\alpha+\rho$ on the circle.
\end{remark}

%=======================================================================
\section{Numerics}\label{sec:numerics}
%=======================================================================

Every certified number is read against a dense computation on the same matrix: the floor of
Theorem~\ref{thm:assembly} against $\sigmin(A-zI)$ from a floating-point singular value
decomposition, the count of Proposition~\ref{prop:count} against the eigenvalues of a dense
factorisation, and $\bar g$ against $\norm{G_{22}}$ formed and normed densely. The dense values carry
rounding and are not themselves certified. The eigenvalues and singular values of
Section~\ref{sec:numnecessary} are formed for every matrix, $\norm{G_{22}}$ and $\sigmin(A-zI)$ for
$n\le2048$, and the eigenvalues and $\sigmin(A-zI)$ of the operator of Section~\ref{sec:numbz} at
$k=2^{12}$. Where a comparison is made, we report whether a certified bound fell on the wrong side of
the dense value, which is how the certificate can be caught being wrong.

The computations use the Julia package RigorousPseudospectra.jl~\cite{RigorousPseudospectra2026},
which implements the certificate on the ball arithmetic of BallArithmetic.jl~\cite{BallArithmetic2023};
the experiments of Sections~\ref{sec:numnecessary} to~\ref{sec:numcert} ran at commit
\texttt{d9cd05c} of the former, which differs from the commit cited only in the drivers of
Section~\ref{sec:numbz}, in drivers the paper does not use, and in a check on the isolation discs.
Its function \texttt{helper\_diagnose} reports in floating point, before a run and at a radius the
user chooses, the quantities the run will certify: the number of Ritz values outside the radius, the
least frame dimension Proposition~\ref{prop:frame} allows, the bound on the far block predicted at
each order of the sweep from the leading singular values of $P_\perp AP_\perp$, computed by a
bidiagonalisation applied matrix-free, with the order to use, and $\gamma_1$, $\gamma_2$ and
$\sig_{11}$ on a few points of the circle with the margin of the hypothesis of
Theorem~\ref{thm:assembly}; nothing it returns enters a certificate. The operator of
Section~\ref{sec:numbz} is assembled with RigorousInvariantMeasures.jl~\cite{RigorousInvariantMeasures2024}.
The experiments below choose
their radii and frames by the rules stated in each subsection. The outputs of every experiment, with each certified
frame stored so that its certificate can be checked again without rebuilding it, are available as
a dataset on the Harvard Dataverse~\cite{NisoliData2026}.

All the runs were made on a 64-core AMD Ryzen Threadripper PRO 9985WX. The matrices are the Brusselator matrices \texttt{rdb200},
\texttt{rdb800l}, \texttt{rdb2048} and \texttt{rdb5000}, the power system \texttt{qh1484}, the Olmstead matrix
\texttt{olm1000} and the symmetric \texttt{nos3} from Matrix Market and SuiteSparse, and
\texttt{grcar500}, \texttt{kahan500}, \texttt{airy200}, \texttt{davies200},
\texttt{convdiff\_cheb200}, \texttt{orrsom200}, \texttt{landau400} and \texttt{hatano2000} from the
EigTool gallery, the four complex ones used as they stand without a real embedding. The four named
for a size of $200$ that the gallery builds at $n=199$ are reported at $199$.

\subsection{What a frame can reach}\label{sec:numnecessary}

Before any certificate is run, Proposition~\ref{prop:frame} says where none can succeed. For a frame
of dimension $d$, $\bar g\ge\norm{G_{22}}\ge\sv_{2d+1}(A)$, so at a radius $R$ the hypothesis
$\bar g<R$ needs $d\ge d_{\min}(R)$, the least $d$ with $\sv_{2d+1}(A)<R$. Placing
$R=(\abs{\lambda_N}\abs{\lambda_{N+1}})^{1/2}$ in each gap between consecutive distinct eigenvalue
moduli, with $N$ eigenvalues outside, Figure~\ref{fig:necessary} plots $d_{\min}$ against $N$ for
every matrix, from a dense eigendecomposition and singular value decomposition of each; the fifteen
took $45$ seconds together. The condition is necessary only: a frame of dimension $d_{\min}$ need not
bring $\bar g$ below $R$, and the couplings and $\sig_{11}$ on the circle decide the rest. At a point
above the diagonal $d_{\min}=N$, every frame that can pass carries more directions than there are
eigenvalues outside, so its compression has eigenvalues inside the disc; these are the radii at
which the separation of Section~\ref{sec:related} cannot be checked through a bound on
$\norm{G_{22}}$, and at which the Main Theorem has to allow them.

\begin{figure}[!tbp]
\centering
\includegraphics[width=\textwidth]{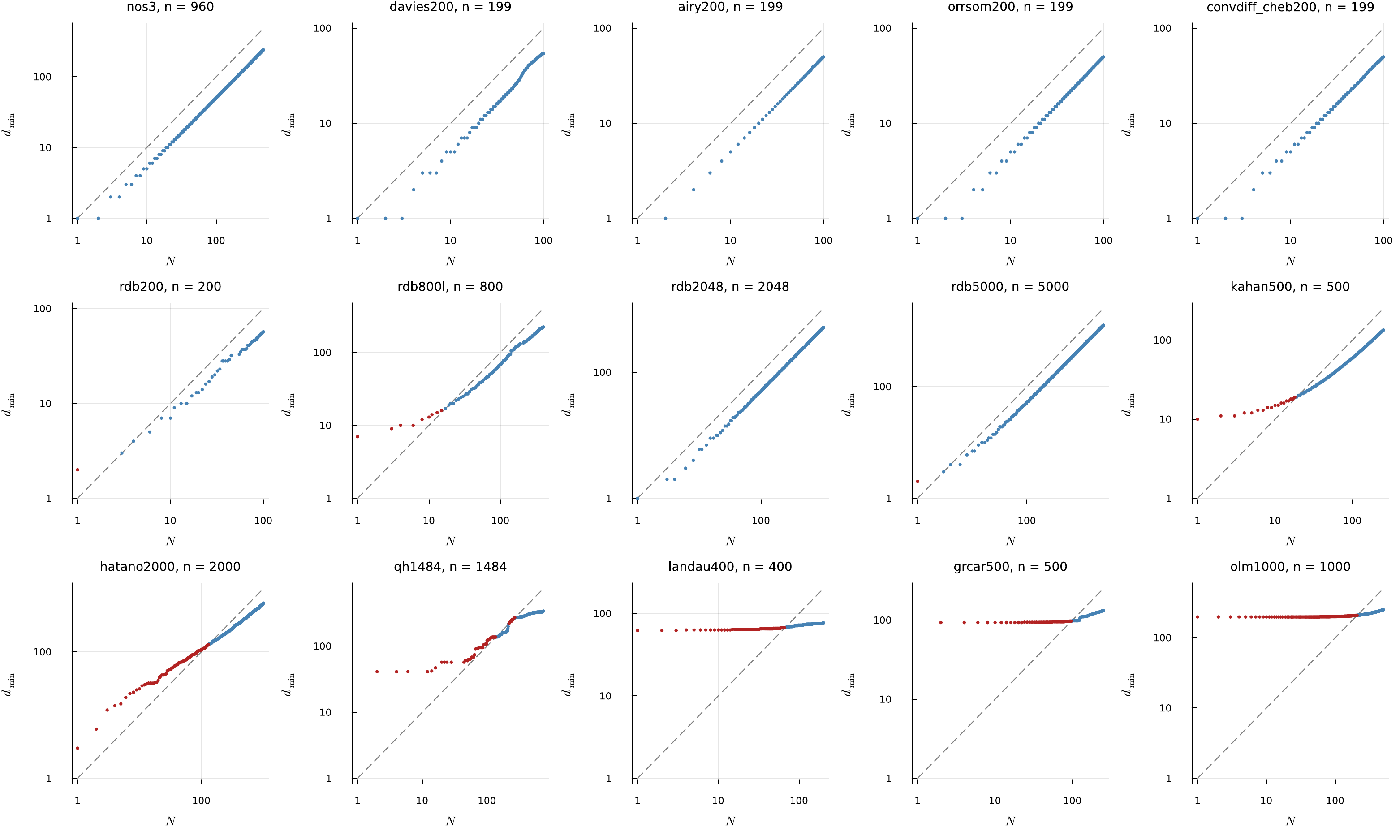}
\caption{The least frame dimension $d_{\min}$ at which $\bar g<R$ is possible, against the number
$N$ of eigenvalues outside the radius $R$ placed in the gap between the $N$-th and $(N+1)$-th
eigenvalue moduli, for $N\le n/2$; points above the dashed diagonal $d_{\min}=N$ in red. Points with
$d_{\min}=0$ are not drawn.}\label{fig:necessary}
\end{figure}

\begin{table}[!tbp]
\centering
\caption{The necessary condition of Figure~\ref{fig:necessary} over the gaps with $N\le n/4$: their
number, the number at which $d_{\min}\le N$, the least $N$ at which $d_{\min}\le N$, the median of
$d_{\min}/N$, and the largest excess $\max(0,d_{\min}-N)$.}\label{tab:necessary}
\begin{tabular}{lrrrrrr}
\toprule
operator & $n$ & gaps & $d_{\min}\le N$ & first $N$ & median $d_{\min}/N$ & max excess\\
\midrule
\texttt{nos3}              & $960$  & $240$ & $240$ & $1$   & $0.50$ & $0$\\
\texttt{davies200}         & $199$  & $49$  & $49$  & $1$   & $0.50$ & $0$\\
\texttt{airy200}           & $199$  & $24$  & $24$  & $2$   & $0.50$ & $0$\\
\texttt{orrsom200}         & $199$  & $49$  & $49$  & $1$   & $0.50$ & $0$\\
\texttt{convdiff\_cheb200} & $199$  & $47$  & $47$  & $1$   & $0.50$ & $0$\\
\texttt{rdb2048}           & $2048$ & $256$ & $256$ & $1$   & $0.50$ & $0$\\
\texttt{rdb5000}           & $5000$ & $625$ & $624$ & $3$   & $0.51$ & $1$\\
\texttt{rdb200}            & $200$  & $25$  & $24$  & $3$   & $0.70$ & $1$\\
\texttt{rdb800l}           & $800$  & $97$  & $88$  & $17$  & $0.71$ & $6$\\
\texttt{kahan500}          & $500$  & $125$ & $107$ & $19$  & $0.66$ & $9$\\
\texttt{hatano2000}        & $2000$ & $500$ & $371$ & $130$ & $0.79$ & $24$\\
\texttt{qh1484}            & $1484$ & $171$ & $96$  & $138$ & $0.98$ & $39$\\
\texttt{landau400}         & $400$  & $100$ & $34$  & $67$  & $1.31$ & $61$\\
\texttt{grcar500}          & $500$  & $62$  & $14$  & $98$  & $1.51$ & $91$\\
\texttt{olm1000}           & $1000$ & $250$ & $43$  & $208$ & $1.60$ & $196$\\
\bottomrule
\end{tabular}
\end{table}

Table~\ref{tab:necessary} gives the same data over the gaps with $N\le n/4$. On \texttt{nos3},
\texttt{davies200}, \texttt{airy200}, \texttt{orrsom200}, \texttt{convdiff\_cheb200} and
\texttt{rdb2048}, $d_{\min}\le N$ at every gap, with $d_{\min}$ close to $N/2$; on \texttt{rdb5000}
and \texttt{rdb200} the one exception is at the top of the spectrum; on \texttt{rdb800l} and
\texttt{kahan500} the points lie above the diagonal at gaps up to $N=15$ and $N=18$, and on or below
it from $N=17$ and $N=19$ on. On
\texttt{hatano2000} they lie above it at gaps up to $N=129$ and on or below it from $N=130$ on; on
\texttt{qh1484} above it up to $N=137$, on or below it from $N=138$ to $216$, and above it again at
$19$ gaps between $N=218$ and $268$, by at most $7$; and on \texttt{landau400}, \texttt{grcar500} and
\texttt{olm1000} $d_{\min}$ stays between $62$ and $67$, $93$ and $97$, and $197$ and $208$ from the
first gap until $N$ reaches it.

\subsection{Frames of one side and of two}\label{sec:numsides}

Proposition~\ref{prop:frame} bounds $\gamma_1\gamma_2$ by $\sv_{k_s+1}^{2}$ when both sides of singular
vectors are in the frame and by $\sv_{k_s+1}\sv_1$ when one is, and the second can also be small. With
the dense eigenvectors of modulus $>R$ and the leading $k_s$ dense singular vectors on the right, on the
left or on both sides, Table~\ref{tab:sides} gives, over the gaps with $N\le n/4$, the number of radii
at which some $k_s$, found by doubling from $k_s=1$ and bisecting, makes the hypothesis hold, with
$\norm{G_{22}}$, the couplings and $\sig_{11}$ formed densely and $\sig_{11}$ minimised over $64$
points of the circle, and the median $\dim\Qs/N$ at those radii. Nothing in this table is certified,
and the frame of the eigenvectors alone, common to the three, is not included. The cap
$\dim\Qs\le n/2$ leaves a two-sided frame at most $(n/2-N)/2$ vectors per side, against $n/2-N$ for a
one-sided one.

\begin{table}[!tbp]
\centering
\caption{Frames of right, left and two-sided singular vectors, with the dense eigenvectors of modulus
$>R$, over the gaps with $N\le n/4$: the number of radii at which, for some $k_s$,
$\norm{G_{22}}<R$ and $\min_{\abs z=R}\sigmin(G_{11}-zI)\,(R-\norm{G_{22}})>\norm{G_{12}}\norm{G_{21}}$,
with the norms formed densely and the minimum taken over $64$ points of the circle, none of it
certified; and the median $\dim\Qs/N$ at those radii. On \texttt{rdb2048}, \texttt{qh1484} and
\texttt{hatano2000} the radii are every gap with $N\le16$ and those at $48$ indices spread
geometrically over the others up to $n/4$, of which $37$, $36$ and $39$ remain after duplicates are
removed;
\texttt{rdb5000} is not included.}\label{tab:sides}
\begin{tabular}{lrrrrrrr}
\toprule
 & & \multicolumn{2}{c}{right} & \multicolumn{2}{c}{left} & \multicolumn{2}{c}{both}\\
operator & radii & met & $\dim\Qs/N$ & met & $\dim\Qs/N$ & met & $\dim\Qs/N$\\
\midrule
\texttt{nos3}              & $240$ & $240$ & $1.00$ & $240$ & $1.00$ & $240$ & $1.00$\\
\texttt{davies200}         & $49$  & $32$  & $1.87$ & $33$  & $1.74$ & $28$  & $2.00$\\
\texttt{airy200}           & $24$  & $19$  & $1.62$ & $21$  & $1.58$ & $16$  & $1.75$\\
\texttt{orrsom200}         & $49$  & $22$  & $2.17$ & $4$   & $8.88$ & $17$  & $2.50$\\
\texttt{convdiff\_cheb200} & $47$  & $25$  & $2.04$ & $25$  & $1.92$ & $21$  & $2.00$\\
\texttt{rdb200}            & $25$  & $0$   & ---    & $0$   & ---    & $15$  & $2.82$\\
\texttt{rdb800l}           & $97$  & $0$   & ---    & $0$   & ---    & $56$  & $2.92$\\
\texttt{landau400}         & $100$ & $58$  & $5.36$ & $57$  & $4.72$ & $0$   & ---\\
\texttt{kahan500}          & $125$ & $0$   & ---    & $0$   & ---    & $0$   & ---\\
\texttt{grcar500}          & $62$  & $0$   & ---    & $0$   & ---    & $0$   & ---\\
\texttt{olm1000}           & $250$ & $0$   & ---    & $0$   & ---    & $0$   & ---\\
\texttt{rdb2048}           & $46$  & $33$  & $1.04$ & $33$  & $1.04$ & $43$  & $1.08$\\
\texttt{qh1484}            & $42$  & $5$   & $16.4$ & $4$   & $30.3$ & $19$  & $18.1$\\
\texttt{hatano2000}        & $55$  & $0$   & ---    & $0$   & ---    & $0$   & ---\\
\bottomrule
\end{tabular}
\end{table}

On \texttt{nos3}, which is symmetric, the three coincide. On the four matrices of size $199$ the
two one-sided frames together meet the hypothesis at every radius at which the two-sided one does and
at $5$ or $6$ more; on \texttt{davies200}, \texttt{airy200} and \texttt{convdiff\_cheb200} one of them
alone does, while on \texttt{orrsom200} the right-sided frame misses one of the radii of the two-sided
one and the left-sided frame meets only $4$. The smaller of the one-sided frames has the smaller
$\dim\Qs$ at all of those radii except three on \texttt{orrsom200}. On \texttt{rdb200}
and \texttt{rdb800l} only the two-sided frame meets it, at $15$ and $56$ radii; on \texttt{landau400}
only the one-sided frames do, at $58$ and $57$ radii. On \texttt{rdb2048} and \texttt{qh1484} the
two-sided frame meets it at every radius at which a one-sided frame does and at $10$ and $13$ more,
and no one-sided frame is smaller there; on \texttt{kahan500}, \texttt{grcar500}, \texttt{olm1000}
and \texttt{hatano2000} none does.

\subsection{The far block, certified and measured}\label{sec:numfar}

For every gap radius of Section~\ref{sec:numnecessary} with $N\le\lfloor n/2\rfloor+1$, the frame of
Definition~\ref{def:Q} was built with $k_s$ singular vectors per side: first $k_s=0$, the
eigendirections alone; then $k_s=K=\#\{\sv_i(A)\ge R\}$, doubled up to $(n/2-N)/2$ until the
certified $\bar g$ of Section~\ref{sec:fpsweep}, at order $8$, fell below $R$; then bisection
between the last frame that failed and the first that passed, as in Section~\ref{sec:depth}; at the
largest one or two $N$ the cap leaves the eigendirections alone. Table~\ref{tab:gbar} compares
$\bar g$ on the frames reported with $\norm{G_{22}}$ formed densely: no ratio is below one, the median
by operator lies between $1.000$ and $1.185$ and the largest single ratio is $1.244$, while the bound
at order one alone has medians between $1.14$ and $18.2$.

\begin{table}[!tbp]
\centering
\renewcommand{\thetable}{\arabic{table}a}\renewcommand{\theHtable}{\arabic{table}a}
\caption{The certified $\bar g$ at order $8$ against the dense $\norm{G_{22}}$, over one frame per gap
radius with $N\le\lfloor n/2\rfloor+1$: the number of radii, the number at which a frame with
$\bar g<R$ was found within the cap, and the median and largest $\bar g/\norm{G_{22}}$, with the
median of the bound at order one, the Frobenius norm. On \texttt{hatano2000}, \texttt{rdb2048} and
\texttt{rdb5000} the radii are every gap with $N\le16$ and those at $48$ indices spread geometrically
over the others up to $n/2$, of which $41$, $39$ and $41$ remain after duplicates are removed; on \texttt{rdb5000}
$\norm{G_{22}}$ was not formed.}\label{tab:gbar}
\begin{tabular}{lrrrrr}
\toprule
operator & radii & $\bar g<R$ & median & largest & median, order one\\
\midrule
\texttt{landau400}         & $201$ & $0$   & $1.000$ & $1.075$ & $1.14$\\
\texttt{qh1484}            & $357$ & $66$  & $1.016$ & $1.068$ & $2.90$\\
\texttt{convdiff\_cheb200} & $88$  & $59$  & $1.042$ & $1.066$ & $4.03$\\
\texttt{airy200}           & $58$  & $36$  & $1.052$ & $1.084$ & $4.34$\\
\texttt{davies200}         & $93$  & $51$  & $1.053$ & $1.073$ & $4.58$\\
\texttt{orrsom200}         & $100$ & $69$  & $1.059$ & $1.084$ & $4.56$\\
\texttt{rdb200}            & $51$  & $11$  & $1.116$ & $1.147$ & $5.62$\\
\texttt{kahan500}          & $251$ & $55$  & $1.127$ & $1.163$ & $10.0$\\
\texttt{olm1000}           & $494$ & $0$   & $1.137$ & $1.203$ & $8.07$\\
\texttt{nos3}              & $481$ & $127$ & $1.150$ & $1.210$ & $13.0$\\
\texttt{grcar500}          & $125$ & $0$   & $1.154$ & $1.191$ & $11.9$\\
\texttt{hatano2000}        & $57$  & $38$  & $1.162$ & $1.244$ & $17.6$\\
\texttt{rdb800l}           & $195$ & $31$  & $1.162$ & $1.196$ & $11.0$\\
\texttt{rdb2048}           & $48$  & $36$  & $1.185$ & $1.213$ & $18.2$\\
\texttt{rdb5000}           & $50$  & $39$  & ---     & ---     & ---\\
\bottomrule
\end{tabular}
\end{table}

Since $\bar g$ is within a factor $1.244$ of $\norm{G_{22}}$, whether a frame passes depends on
$\norm{G_{22}}$ itself, which
Proposition~\ref{prop:frame} places between $\sv_{2d+1}(A)$, for $d=\dim\Qs$, and, for a frame
containing the exact leading $k_s$ singular vectors on each side, $\sv_{k_s+1}(A)$.
Table~\ref{tab:framesv} compares, at the same radius and $k_s$, the frame of Definition~\ref{def:Q}
used above, whose singular vectors come from a Golub--Kahan--Lanczos bidiagonalisation of length
$2k_s+60$, with the frame whose singular vectors come from a dense singular value decomposition; its
caption records two further frames, the bidiagonalisation run to length $n$ and the dense
eigenvectors in place of the Arnoldi ones. On \texttt{davies200}, \texttt{olm1000} and
\texttt{qh1484}, $\norm{G_{22}}$ sits at or below $\sv_{k_s+1}$, often on it, and far above
$\sv_{2d+1}$, and the frames agree to within a few per cent except on \texttt{olm1000} at $N=250$. On
\texttt{landau400}, where $\sv_{k_s+1}=1$ at every point tested, at the three radii of the table the
dense singular vectors bring $\norm{G_{22}}$ below $R$ at $\dim\Qs=199$ and $200$, and the
bidiagonalisation, of either length, does not; at $N=1$ both do.

\begin{table}[!tbp]
\centering
\addtocounter{table}{-1}\renewcommand{\thetable}{\arabic{table}b}\renewcommand{\theHtable}{\arabic{table}b}
\caption{$\norm{G_{22}}$, formed densely, for two frames at the same radius and $k_s$, both with Arnoldi
eigendirections: singular vectors from a Golub--Kahan--Lanczos bidiagonalisation of length $2k_s+60$,
and from a dense singular value decomposition. $d$ is the dimension of the first frame; the second
is within $3$ of it except on \texttt{olm1000} at $N=250$, where it is $309$. A bidiagonalisation of length $n$ gives the same values except
$376$ on \texttt{davies200} at $N=30$ and $7.72\cdot10^{4}$ on \texttt{olm1000} at $N=250$; dense
eigenvectors with the dense singular vectors give the same values except $4.59\cdot10^{4}$ on
\texttt{olm1000} at $N=250$.}\label{tab:framesv}
\begin{tabular}{lrrrrrrrr}
\toprule
operator & $N$ & $R$ & $k_s$ & $d$ & bidiag. & dense SVD & $\sv_{2d+1}$ & $\sv_{k_s+1}$\\
\midrule
\texttt{davies200} & $11$ & $4819$ & $10$ & $31$ & $1180$ & $1180$ & $257$ & $4820$\\
\texttt{davies200} & $30$ & $743.8$ & $30$ & $60$ & $362$ & $376$ & $133$ & $703$\\
\texttt{olm1000} & $1$ & $1.016\cdot10^{4}$ & $197$ & $395$ & $6.09\cdot10^{4}$ & $6.09\cdot10^{4}$ & $0.552$ & $6.09\cdot10^{4}$\\
\texttt{olm1000} & $250$ & $5077$ & $125$ & $327$ & $7.16\cdot10^{4}$ & $7.72\cdot10^{4}$ & $0.552$ & $7.85\cdot10^{4}$\\
\texttt{qh1484} & $2$ & $6.675\cdot10^{11}$ & $82$ & $165$ & $6.67\cdot10^{11}$ & $6.67\cdot10^{11}$ & $9.35\cdot10^{7}$ & $6.67\cdot10^{11}$\\
\texttt{qh1484} & $94$ & $1.353\cdot10^{9}$ & $214$ & $466$ & $1.31\cdot10^{9}$ & $1.31\cdot10^{9}$ & $377$ & $1.31\cdot10^{9}$\\
\texttt{landau400} & $21$ & $0.838$ & $89$ & $199$ & $0.901$ & $0.0957$ & $2.8\cdot10^{-17}$ & $1$\\
\texttt{landau400} & $41$ & $0.3424$ & $79$ & $199$ & $0.887$ & $0.0612$ & $2.8\cdot10^{-17}$ & $1$\\
\texttt{landau400} & $60$ & $0.04817$ & $70$ & $200$ & $0.735$ & $0.0275$ & $0$ & $1$\\
\bottomrule
\end{tabular}
\end{table}

\subsection{The certificate}\label{sec:numcert}

At every radius at which Section~\ref{sec:numfar} found $\bar g<R$ the frame is certified in full:
the far-block bound by the sweep at order $8$, $\gamma_1$ and $\gamma_2$ by the Gram matrices of
Section~\ref{sec:coupling}, the number of eigenvalues of $G_{11}$ outside $R$ by Theorem~2.2
of~\cite{Rump2022a} and Lemma~\ref{lem:homotopy}, $\sig_{11}$ by one verified singular value
decomposition per point, and the
hypothesis on the whole circle by the cover of Remark~\ref{rem:cover}, from $64$ initial arcs. When
the cover fails, $k_s$ is doubled up to $(n/2-N)/2$ and bisected back from the first frame that
passes. With the bisection capped at depth $64$ (Remark~\ref{rem:cover}), a cover fails where the
hypothesis fails at a sample point, where an arc cannot be split in floating point, or at the cap;
every failure in Table~\ref{tab:cert} is of the first kind.
The frame reported at each radius is the one at which the search ended, and it is stored with its
certificate in the data deposit. On \texttt{qh1484} the run was stopped after $92$ hours, when the
search had ended at $6$ of the $17$ radii of its sample; Table~\ref{tab:cert} reports those six.

\begin{table}[!tbp]
\centering
\caption{The full certificate at the radii of Table~\ref{tab:gbar} with $\bar g<R$, sampled on four
operators as stated below: their number, the
number at which the hypothesis is certified on the circle and the count outside $R$ is certified
(the two coincide here), how many of those at the frame of Section~\ref{sec:numfar}, the range of
$N$ certified, the median $\dim\Qs/N$, and the median of the least lower bound the cover certifies on
the circle over the least dense $\sigmin(A-zI)$ at $64$ points of the circle, a ratio that is at most
$0.70$ wherever the dense value is formed. At every radius certified the certified count equals the dense count $N$. On \texttt{rdb2048}, \texttt{hatano2000}, \texttt{rdb5000} and \texttt{qh1484}
the radii are every feasible one with $N\le16$ and a sample of the others spread geometrically, of
at most $48$ on the first three and at most $12$ on \texttt{qh1484}, where only the six radii $N=2,12,14,16,28,118$
at which the search ended are reported; on \texttt{rdb5000} no dense $\sigmin$ was
formed.}\label{tab:cert}
\begin{tabular}{lrrrrrr}
\toprule
operator & radii & certified & at that frame & $N$ & $\dim\Qs/N$ & floor/$\sigmin$\\
\midrule
\texttt{nos3}              & $127$ & $34$ & $0$  & $1$--$34$ & $4.3$  & $2.4\cdot10^{-1}$\\
\texttt{davies200}         & $51$  & $18$ & $14$ & $1$--$30$ & $1.0$  & $1.1\cdot10^{-2}$\\
\texttt{airy200}           & $36$  & $16$ & $15$ & $2$--$32$ & $1.0$  & $1.6\cdot10^{-2}$\\
\texttt{orrsom200}         & $69$  & $19$ & $15$ & $1$--$30$ & $1.0$  & $1.9\cdot10^{-5}$\\
\texttt{convdiff\_cheb200} & $59$  & $21$ & $13$ & $1$--$28$ & $1.0$  & $1.6\cdot10^{-2}$\\
\texttt{rdb200}            & $11$  & $11$ & $11$ & $1$--$19$ & $4.8$  & $1.2\cdot10^{-1}$\\
\texttt{rdb800l}           & $31$  & $8$  & $0$  & $1$--$13$ & $42.1$ & $1.8\cdot10^{-1}$\\
\texttt{kahan500}          & $55$  & $0$  & ---  & ---       & ---    & ---\\
\texttt{qh1484}            & $6$   & $5$  & $2$  & $2$--$28$ & $21.0$ & $8.6\cdot10^{-2}$\\
\texttt{rdb2048}           & $32$  & $1$  & $0$  & $1$       & $968$  & $3.2\cdot10^{-2}$\\
\texttt{hatano2000}        & $37$  & $0$  & ---  & ---       & ---    & ---\\
\texttt{rdb5000}           & $34$  & $3$  & $0$  & $1$--$4$  & $831$  & ---\\
\bottomrule
\end{tabular}
\end{table}

On the four matrices of size $199$ the certificate holds up to $N$ between $28$ and $32$, with
$\dim\Qs$ close to $N$. On \texttt{rdb800l} the eight radii certified, $N\le13$, all lie where
$d_{\min}>N$ in Table~\ref{tab:necessary}, with a median $\dim\Qs/N$ of $42.1$. On \texttt{rdb2048}
and \texttt{rdb5000} the radii certified, $N=1$ and $N=1,3,4$, need frames of $968$ and of $2494$
columns, close to the cap $n/2$; on \texttt{qh1484} the five radii certified, $N=2$ to $28$, need
frames of $252$ or $572$ columns; on \texttt{kahan500} and \texttt{hatano2000} the hypothesis fails
at a sample point at every radius.

Figure~\ref{fig:certmap} shows, for \texttt{davies200} at $N=30$, the larger of the lower bounds of
Theorem~\ref{thm:assembly} and Proposition~\ref{prop:blockdom} on $\sigmin(A-zI)$ over the plane, at the frame of
Table~\ref{tab:cert}: radius $R=743.8$, $\dim\Qs=62$, $\bar g=358.6$ from the sweep of order $8$, and $\sig_{11}$ by one verified
singular value decomposition at each of $300\times300$ grid points in each panel. Over the whole
plane the bound exists at every grid point, and its ratio to the dense $\sigmin(A-zI)$ has median
$0.999$ and least value $0.697$; in the zoom it exists at $84390$ of the $90000$ points. At no grid
point does it exceed the dense value.

\begin{figure}[!tbp]
\centering
\includegraphics[width=\textwidth]{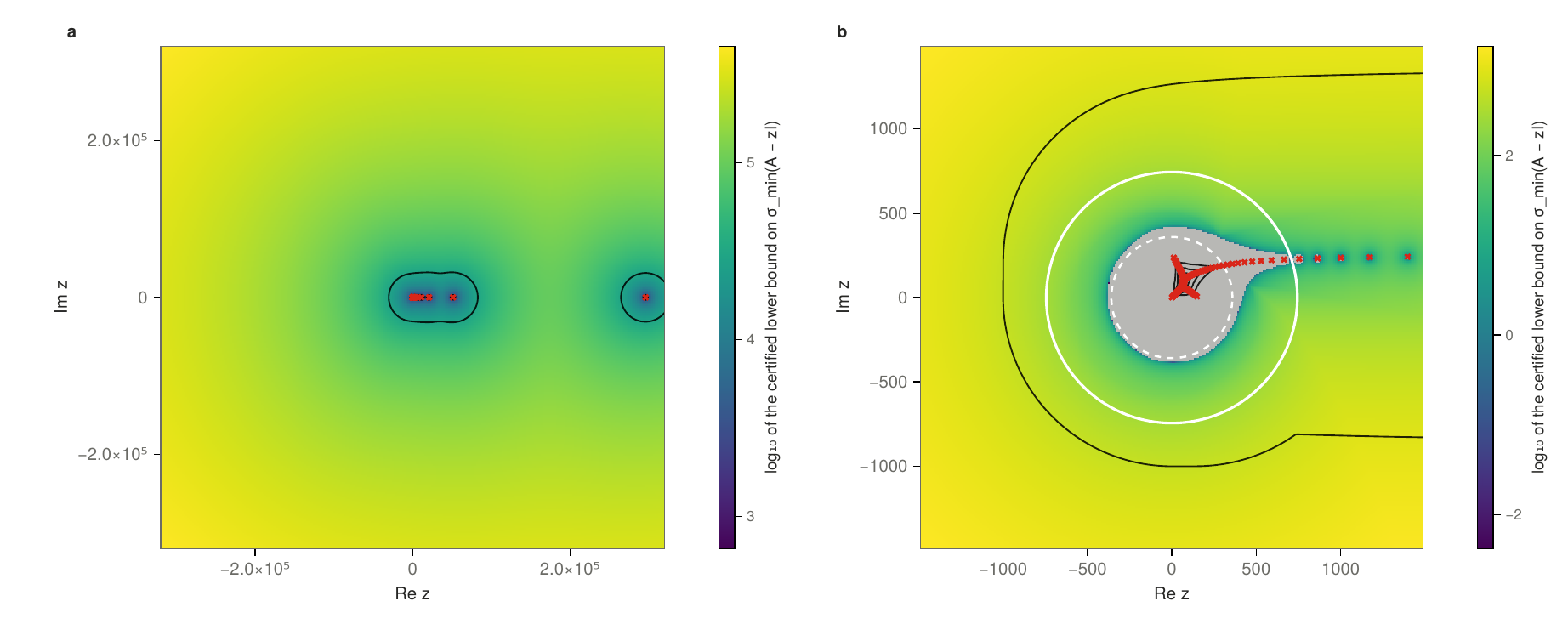}
\caption{\texttt{davies200}, $N=30$: the certified lower bound on $\sigmin(A-zI)$, $\log_{10}$, (a)
over $[-3.2\cdot10^{5},3.2\cdot10^{5}]^2$ and (b) over $[-2R,2R]^2$, each panel on its own colour
scale. Grey: no bound at that point. White: the circle $\abs z=R$, and dashed the circle
$\abs z=\bar g$. Black: contours of the dense $\sigmin(A-zI)$, at $\varepsilon=10^{-3},10^{-2},
10^{-1}$ times the certified upper bound $\bar a$ on $\norm A$ in (a) and at
$10^{-8},10^{-4},10^{-1},10^{3}$ in (b). Red: the eigenvalues of $A$.}\label{fig:certmap}
\end{figure}

Figure~\ref{fig:certmapb} shows the same for \texttt{rdb800l} at $N=8$, a radius at which
$d_{\min}>N$: $R=30.02$, $\dim\Qs=322$, $\bar g=25.36$, on a grid of $200\times200$ points over
$[-34.5,34.5]^2$. The bound exists at $23222$ of the $40000$ points; its ratio to the dense
$\sigmin(A-zI)$ has median $0.296$, and at no grid point does it exceed the dense value.

\begin{figure}[!tbp]
\centering
\includegraphics[width=0.5\textwidth]{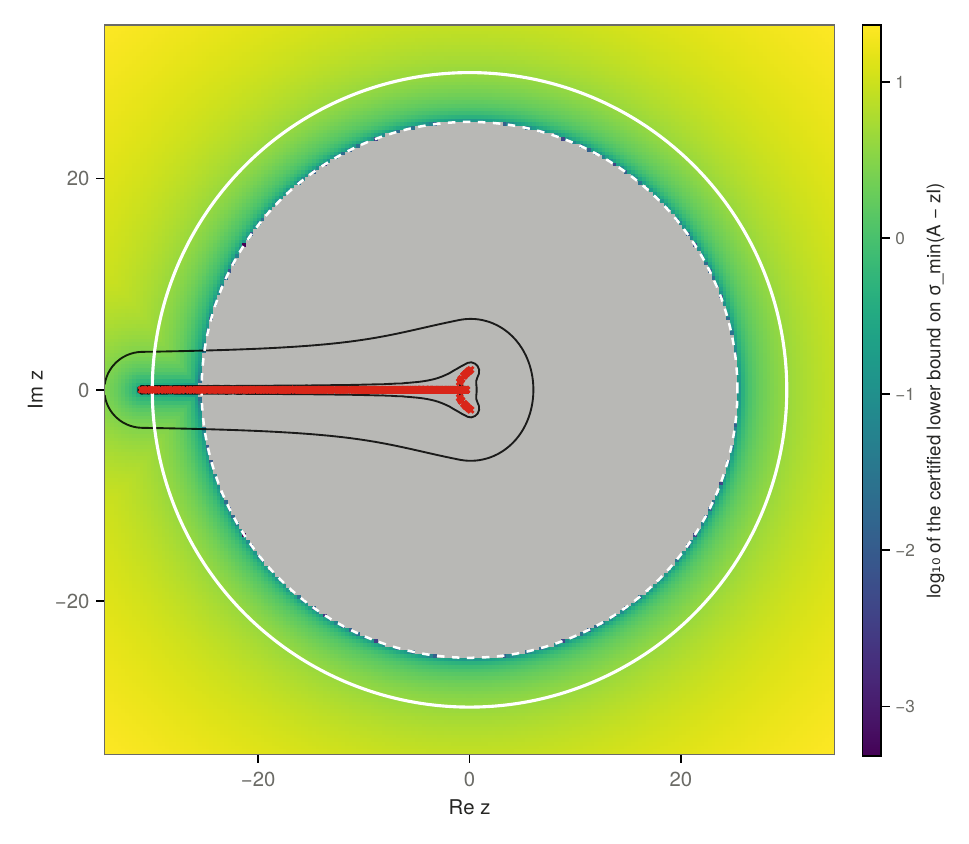}
\caption{\texttt{rdb800l}, $N=8$: the certified lower bound on $\sigmin(A-zI)$, $\log_{10}$, with
the conventions of Figure~\ref{fig:certmap}; contours of the dense $\sigmin(A-zI)$ at
$\varepsilon=10^{-3},10^{-2},10^{-1}$ times $\bar a$.}\label{fig:certmapb}
\end{figure}

\subsection{An operator that is not stored}\label{sec:numbz}

The bounds of this section are bounds on the pseudospectrum and the eigenvalues of the
discretised operator $A=\mathcal NL$ below, a matrix of size $k$, and not on those of the transfer operator it
discretises; the passage to the transfer operator is the argument
of~\cite{BlumenthalEtAl2025,Nisoli2026}, which uses a lower bound on $\sigmin(A-zI)$ of the kind
certified here.

The map is that of the Belousov--Zhabotinsky reaction studied
in~\cite{MatsumotoTsuda1983,GalatoloEtAl2020}. $L$ is the midpoint of the interval Ulam matrix of
the map on a partition of $[0,1]$ into $k$ cells~\cite{GalatoloNisoli2014,GalatoloEtAl2023}, sparse, and $\mathcal N$ is the
convolution with a periodic uniform noise of half-width $\xi$, that is of $\lfloor\xi k\rceil$
cells, both built with RigorousInvariantMeasures.jl~\cite{RigorousInvariantMeasures2024}. The operator $A=\mathcal NL$ is never formed: it is reached through
$y\mapsto \mathcal N(Ly)$ and $y\mapsto L^{*}(\mathcal Ny)$. Its rounding satisfies $\norm{\fl(Ay)-Ay}\le c_A\norm y$,
with $c_A$ bounded from the stencil width of $L$, the half-width of $\mathcal N$ and an upper bound on
$\norm L$, and enters \eqref{eq:roundA} as $\varepsilon_A=c_A/a_{\mathrm{lo}}$, with $a_{\mathrm{lo}}$ a
certified lower bound on $\norm A$ from one application to the constant vector. We take $\xi=0.1$, at $k=2^{12}$, where the dense
eigenvalues and $\sigmin(A-zI)$ can still be formed for comparison, and at $k=2^{16}$, where they are
not, and the radii are placed in the gaps of the dense eigenvalues at $k=2^{12}$ and of the Ritz values
of an Arnoldi process of length $300$ at $k=2^{16}$. At each radius with $N\le16$ the frame is
searched as in Section~\ref{sec:depth}, from $k_s=0$ and then $k_s=8$, doubled up to $512$; the count
is certified on the circle, and around each component of the
enclosure of the eigenvalues of $G_{11}$ outside $R$ the smallest disc on a geometric ladder of $12$
radii is sought whose boundary passes the hypothesis of Proposition~\ref{prop:count}.

\begin{table}[!tbp]
\centering
\caption{The certificate and the isolation discs on the Belousov--Zhabotinsky operator $A=\mathcal NL$ with
$\xi=0.1$, at the radii with $N\le7$ eigenvalues outside: the frame found, $\bar g$ at order $8$,
the certified count outside $R$, the largest radius of the isolation discs, and the wall clock of
the search and the discs. At $k=2^{12}$ every dense eigenvalue of modulus greater than $R$ lies in a
disc holding the number the disc certifies. A dash: no frame tried passed, with $k_s=0,8,16,\dots,512$;
at the smallest the certified $\bar g$ is not below $R$ ($k_s\le8$ at $N=5$, $k_s\le16$ at $N=7$), and
at the others the hypothesis fails at a sample point of the circle.}\label{tab:bz}
\begin{tabular}{rrrrrrrr}
\toprule
$k$ & $N$ & $R$ & $\dim\Qs$ & $\bar g$ & count & largest disc & time (s)\\
\midrule
$2^{12}$ & $1$ & $0.752$  & $7$    & $0.288$   & $1$ & $0.248$  & $37$\\
         & $3$ & $0.307$  & $17$   & $0.132$   & $3$ & $0.259$  & $5$\\
         & $4$ & $0.120$  & $94$   & $0.0271$  & $4$ & $0.0469$ & $70$\\
         & $5$ & $0.0734$ & $363$  & $0.00752$ & $5$ & $0.0125$ & $1950$\\
         & $7$ & $0.0513$ & $677$  & $0.00399$ & $7$ & $0.0346$ & $30743$\\
$2^{16}$ & $1$ & $0.752$  & $7$    & $0.289$   & $1$ & $0.248$  & $714$\\
         & $3$ & $0.309$  & $17$   & $0.133$   & $3$ & $0.256$  & $710$\\
         & $4$ & $0.115$  & $102$  & $0.0251$  & $4$ & $0.0541$ & $1992$\\
         & $5$ & $0.0760$ & ---    & ---       & --- & ---      & $2877$\\
         & $7$ & $0.0591$ & ---    & ---       & --- & ---      & $2900$\\
\bottomrule
\end{tabular}
\end{table}

Table~\ref{tab:bz} gives the outcome. At $k=2^{12}$ the count and the discs are certified at every
radius with $N\le7$, and each count equals the number of dense eigenvalues outside $R$. At
$k=2^{16}$ they are certified for $N\le4$, with frames of $7$, $17$ and $102$ columns against
$7$, $17$ and $94$ at $2^{12}$; for $N=5$ and $7$ no frame up to $k_s=512$, $\dim\Qs=1029$ and
$1031$, passes, the larger ones failing at a sample point of the circle, where at $2^{12}$ frames of
$363$ and $677$ columns pass. The disc certified about the leading eigenvalue at $2^{12}$, $N=7$, has
radius $2.85\cdot10^{-5}$ and a centre within $1.2\cdot10^{-14}$ of $1$.

Figure~\ref{fig:bzmap} shows the larger of the lower bounds of Theorem~\ref{thm:assembly} and
Proposition~\ref{prop:blockdom} on $\sigmin(A-zI)$ at $N=4$
for both sizes, on the frames of Table~\ref{tab:bz}, over $[-1.15,1.15]^2$: the bound exists at
$158880$ of $160000$ grid points at $2^{12}$ and at $158956$ of $160000$ at $2^{16}$, the grid taking
$487$ and $1469$ seconds on $12$ threads.
At $2^{12}$, at the $118$ of $120$ points of a coarser grid at which it exists, its ratio to the dense $\sigmin(A-zI)$ has median
$0.926$ and least value $0.570$, and it never exceeds the dense value. The four isolation discs are
centred at $1$, with radius $7\cdot10^{-4}$ at both sizes; at $-0.082\pm0.560i$, with radius
$1.2\cdot10^{-2}$ at both sizes; and at $0.167$ and $0.169$, with radii $0.047$ and $0.054$ at $2^{12}$
and $2^{16}$.

\begin{figure}[!tbp]
\centering
\includegraphics[width=\textwidth]{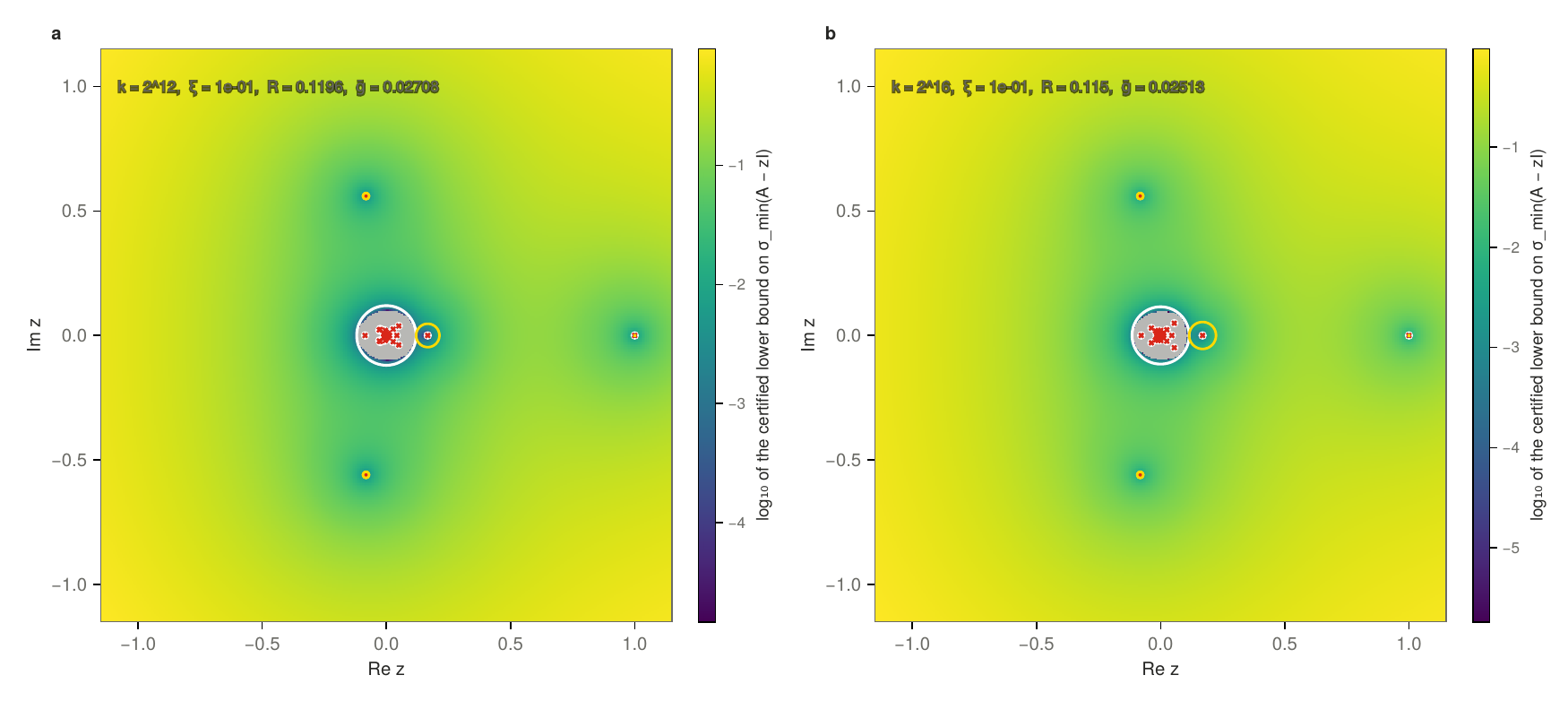}
\caption{The Belousov--Zhabotinsky operator with $\xi=0.1$ at $N=4$, (a) $k=2^{12}$ and (b)
$k=2^{16}$: the certified lower bound on $\sigmin(A-zI)$, $\log_{10}$, each panel on its own colour
scale. Grey: no bound at that point. White: the circle $\abs z=R$, and dashed the circle
$\abs z=\bar g$. Yellow: the isolation discs. Red: the dense eigenvalues in (a), the leading Ritz
values in (b).}\label{fig:bzmap}
\end{figure}

%=======================================================================
\section{Discussion}\label{sec:discussion}
%=======================================================================

Proposition~\ref{prop:frame} ties the frames that can pass to the singular-value decay of $A$.
Section~\ref{sec:numnecessary} reads off the decay where no frame can succeed, and
Section~\ref{sec:numcert} where the certificate holds: on the four matrices of size $199$ up to $N$ between $28$ and $32$, with $\dim\Qs$ close to
$N$; on \texttt{rdb800l} at radii where $d_{\min}>N$, with a median $\dim\Qs/N$ of $42.1$; on
\texttt{qh1484} at five of the six radii run, with a median $\dim\Qs/N$ of $21$; on \texttt{rdb2048} and
\texttt{rdb5000} only with frames close to the cap $n/2$; on \texttt{kahan500} and \texttt{hatano2000}
nowhere. The certified $\bar g$ exceeds the measured $\norm{G_{22}}$ by at
most $25\%$ (Table~\ref{tab:gbar}), and the measured $\norm{G_{22}}$ sits at or below
$\sv_{k_s+1}(A)$ and far above the lower bound $\sv_{2d+1}(A)$ of Proposition~\ref{prop:frame}
(Table~\ref{tab:framesv}).

Inside the disc the only bound available without further work is
$\sigmin(G_{22}-zI)\ge\sigmin(G_{22})-\abs{z}$, and $\sigmin(G_{22})$ sits in the tail of the
singular spectrum of $A$, since a compression does not increase singular values and so
$\sigmin(G_{22})\le\sv_{n-k}(A)$ on $\ran P_\perp$. Shifting does not rescue it: for $A$ compact, $A-z_0I$
differs from $-z_0I$ by a compact operator, so by Proposition~\ref{prop:essfloor} no frame brings the
far block of the shifted operator below $\norm{A-z_0I}_{\mathrm{ess}}=\abs{z_0}$, and the disc that
has to be excluded about $z_0$ always contains the origin. The continuation that would reach the
interior is the sweep of Section~\ref{sec:matfree} applied to $(A-z_0I)^{-1}$, which needs certified
solves rather than applications, at $mn$ solves per centre; the inequalities of Section~\ref{sec:norm}
transfer to it unchanged, and we have not carried it out.

A certified isolation circle about an eigenvalue $\lambda$ cannot be shrunk below a radius of about
$\gamma_1\gamma_2/(\abs{\lambda}-\bar g)$, since Theorem~\ref{thm:assembly} needs
$\sig_{11}\sig_{22}>\gamma_1\gamma_2$ and $\sig_{11}$ vanishes at an eigenvalue of $G_{11}$. For the
example of Section~\ref{sec:frame}, $\gamma_1\gamma_2=7.0\cdot10^{-4}$; the coupling
$\gamma_2=\norm{P_\perp AQ}$ vanishes when the kept subspace is $A$-invariant. The bound on the far block depends on the order of the sweep through
the flatness of the singular spectrum of the far block (Proposition~\ref{prop:ladder}).

The cover puts the method on contours, at one verified singular value
decomposition of size $\dim\Qs$ per point; a picture on a grid pays that at every grid point.
Figure~\ref{fig:certmap} is such a picture, $90000$ points at $\dim\Qs=62$ per panel, in $224$ and
$218$ seconds on $16$ threads of a machine shared with other runs. The order of the sweep is fixed at
$8$ in Section~\ref{sec:numerics}.

The cover of Remark~\ref{rem:cover} applies to any contour along which
the distance between consecutive samples is bounded. An arc is accepted when the hypothesis holds at
its samples with the lower bounds reduced by half its length, so long arcs suffice where the
hypothesis holds with room to spare; choosing the contour to make that room uniform along it is
what we have not carried out, and the
circle is enough for the statements above and for the experiments.

\section*{Acknowledgements}
The author thanks W.~Tucker and F.~Poloni for helpful discussions.
The author thanks Prof.~Y.~Sato for access to computational resources. This work was supported by
the Research Institute for Mathematical Sciences, an International Joint Usage/Research Center
located in Kyoto University. The author thanks the Isaac Newton Institute for
Mathematical Sciences, Cambridge, for support and hospitality during the workshop
\emph{Operator approaches to dynamics: new connections} of the programme \emph{Operator methods for
dynamical systems}, where work on this paper was undertaken; this work was supported by EPSRC grant
EP/Z000580/1.
Claude (Anthropic) was used as an assistant for prose editing and code development.

\begin{sloppypar}
I.N.\ was partially supported by the Postgraduate Program in Mathematics at UFRJ,
CAPES--Finance Code 001, CNPq Projeto Universal No.\ 404943/2023-3,
CAPES--PRINT No.\ 88881.311616/2018-00, and CAPES--STINT No.\ 88887.155746/2017-00.
\end{sloppypar}

Any remaining errors are solely the author's responsibility.

%=======================================================================
\bibliographystyle{plain}
\FloatBarrier
\begingroup
\small
\bibliography{refs}

\begin{thebibliography}{10}

\bibitem{Bandtlow2008}
Oscar~F. Bandtlow.
\newblock Resolvent estimates for operators belonging to exponential classes.
\newblock {\em Integral Equations and Operator Theory}, 61(1):21--43, 2008.

\bibitem{BandtlowJenkinson2008}
Oscar~F. Bandtlow and Oliver Jenkinson.
\newblock Explicit eigenvalue estimates for transfer operators acting on spaces
  of holomorphic functions.
\newblock {\em Advances in Mathematics}, 218(3):902--925, 2008.

\bibitem{BauerFike1960}
F.~L. Bauer and C.~T. Fike.
\newblock Norms and exclusion theorems.
\newblock {\em Numerische Mathematik}, 2(1):137--141, 1960.

\bibitem{Behnke1991}
Henning Behnke.
\newblock The calculation of guaranteed bounds for eigenvalues using
  complementary variational principles.
\newblock {\em Computing}, 47:11--27, 1991.

\bibitem{Beyn2012}
Wolf-J\"urgen Beyn.
\newblock An integral method for solving nonlinear eigenvalue problems.
\newblock {\em Linear Algebra and its Applications}, 436(10):3839--3863, 2012.

\bibitem{BlumenthalEtAl2025}
Alex Blumenthal, Isaia Nisoli, and Toby Taylor-Crush.
\newblock A pseudospectral approach to rigorous numerical estimation of
  resonances of transfer operators, 2025.
\newblock arXiv:2507.09021; accepted, Foundations of Computational Mathematics.

\bibitem{Boegli2017}
Sabine B{\"o}gli.
\newblock Convergence of sequences of linear operators and their spectra.
\newblock {\em Integral Equations and Operator Theory}, 88(4):559--599, 2017.

\bibitem{Boegli2018}
Sabine B{\"o}gli.
\newblock Local convergence of spectra and pseudospectra.
\newblock {\em Journal of Spectral Theory}, 8(3):1051--1098, 2018.

\bibitem{Brauer1947}
Alfred Brauer.
\newblock Limits for the characteristic roots of a matrix. {II}.
\newblock {\em Duke Mathematical Journal}, 14(1):21--26, 1947.

\bibitem{BujanovicKressner2021}
Zvonimir Bujanovi\'{c} and Daniel Kressner.
\newblock Norm and trace estimation with random rank-one vectors.
\newblock {\em SIAM Journal on Matrix Analysis and Applications},
  42(1):202--223, 2021.

\bibitem{CardenEmbree2012}
Russell~L. Carden and Mark Embree.
\newblock Ritz value localization for non-{Hermitian} matrices.
\newblock {\em SIAM Journal on Matrix Analysis and Applications},
  33(4):1320--1338, 2012.

\bibitem{Colbrook2024}
Matthew~J. Colbrook.
\newblock The multiverse of dynamic mode decomposition algorithms.
\newblock In {\em Numerical Analysis Meets Machine Learning}, Handbook of
  Numerical Analysis, pages 127--230. Elsevier, 2024.

\bibitem{ColbrookTownsend2024}
Matthew~J. Colbrook and Alex Townsend.
\newblock Rigorous data-driven computation of spectral properties of {K}oopman
  operators for dynamical systems.
\newblock {\em Communications on Pure and Applied Mathematics}, 77(1):221--283,
  2024.

\bibitem{DaviesPlum2004}
E.~Brian Davies and Michael Plum.
\newblock Spectral pollution.
\newblock {\em IMA Journal of Numerical Analysis}, 24(3):417--438, 2004.

\bibitem{Delattre2024}
Blaise Delattre, Quentin Barth\'{e}lemy, and Alexandre Allauzen.
\newblock Spectral norm of convolutional layers with circular and zero
  paddings, 2024.
\newblock arXiv:2402.00240.

\bibitem{Delattre2023}
Blaise Delattre, Quentin Barth\'{e}lemy, Alexandre Araujo, and Alexandre
  Allauzen.
\newblock Efficient bound of {Lipschitz} constant for convolutional layers by
  {Gram} iteration.
\newblock In {\em Proceedings of the 40th International Conference on Machine
  Learning}, volume 202 of {\em PMLR}, pages 7513--7532, 2023.
\newblock arXiv:2305.16173.

\bibitem{DeuflhardEtAl2000}
Peter Deuflhard, Wilhelm Huisinga, Alexander Fischer, and Christof Sch\"utte.
\newblock Identification of almost invariant aggregates in reversible nearly
  uncoupled {Markov} chains.
\newblock {\em Linear Algebra and its Applications}, 315:39--59, 2000.

\bibitem{DeuflhardWeber2005}
Peter Deuflhard and Marcus Weber.
\newblock Robust {Perron} cluster analysis in conformation dynamics.
\newblock {\em Linear Algebra and its Applications}, 398:161--184, 2005.

\bibitem{Dixon1983}
John~D. Dixon.
\newblock Estimating extremal eigenvalues and condition numbers of matrices.
\newblock {\em SIAM Journal on Numerical Analysis}, 20(4):812--814, 1983.

\bibitem{DjurdjevacEtAl2012}
Nata\v{s}a Djurdjevac, Marco Sarich, and Christof Sch\"utte.
\newblock Estimating the eigenvalue error of {Markov} state models.
\newblock {\em Multiscale Modeling \& Simulation}, 10(1):61--81, 2012.

\bibitem{BallArithmetic2023}
Luca Ferranti and Isaia Nisoli.
\newblock {BallArithmetic}.jl.
\newblock \url{https://github.com/JuliaBallArithmetic/BallArithmetic.jl}, 2026.
\newblock Tag \texttt{reverse-ritz-arxiv-v1}, commit \texttt{386fdbf}.

\bibitem{FroylandDellnitz2003}
Gary Froyland and Michael Dellnitz.
\newblock Detecting and locating near-optimal almost-invariant sets and cycles.
\newblock {\em SIAM Journal on Scientific Computing}, 24(6):1839--1863, 2003.

\bibitem{GalatoloEtAl2020}
Stefano Galatolo, Maurizio Monge, and Isaia Nisoli.
\newblock Existence of noise induced order, a computer aided proof.
\newblock {\em Nonlinearity}, 33(9):4237--4276, 2020.

\bibitem{GalatoloEtAl2023}
Stefano Galatolo, Maurizio Monge, Isaia Nisoli, and Federico Poloni.
\newblock A general framework for the rigorous computation of invariant
  densities and the coarse-fine strategy.
\newblock {\em Chaos, Solitons \& Fractals}, 170:113329, 2023.

\bibitem{GalatoloNisoli2014}
Stefano Galatolo and Isaia Nisoli.
\newblock An elementary approach to rigorous approximation of invariant
  measures.
\newblock {\em SIAM Journal on Applied Dynamical Systems}, 13(2):958--985,
  2014.

\bibitem{HMT2011}
Nathan Halko, Per-Gunnar Martinsson, and Joel~A. Tropp.
\newblock Finding structure with randomness: probabilistic algorithms for
  constructing approximate matrix decompositions.
\newblock {\em SIAM Review}, 53(2):217--288, 2011.

\bibitem{HerwigEtAl2025}
April Herwig, Matthew~J. Colbrook, Oliver Junge, Péter Koltai, and Julia
  Slipantschuk.
\newblock Avoiding spectral pollution for transfer operators using residuals,
  2025.
\newblock arXiv:2507.16915.

\bibitem{Higham2002}
Nicholas~J. Higham.
\newblock {\em Accuracy and Stability of Numerical Algorithms}.
\newblock Society for Industrial and Applied Mathematics, Philadelphia, second
  edition, 2002.

\bibitem{Hochstenbach2013}
Michiel~E. Hochstenbach.
\newblock Probabilistic upper bounds for the matrix two-norm.
\newblock {\em Journal of Scientific Computing}, 57(3):464--476, 2013.

\bibitem{ImakuraEtAl2020}
Akira Imakura, Keiichi Morikuni, and Akitoshi Takayasu.
\newblock Verified partial eigenvalue computations using contour integrals for
  {Hermitian} generalized eigenproblems.
\newblock {\em Journal of Computational and Applied Mathematics}, 369:112543,
  2020.

\bibitem{JiaStewart2000}
Zhongxiao Jia and G.~W. Stewart.
\newblock An analysis of the {Rayleigh--Ritz} method for approximating
  eigenspaces.
\newblock {\em Mathematics of Computation}, 70(234):637--647, 2001.

\bibitem{KahanEtAl1982}
W.~Kahan, B.~N. Parlett, and E.~Jiang.
\newblock Residual bounds on approximate eigensystems of nonnormal matrices.
\newblock {\em SIAM Journal on Numerical Analysis}, 19(3):470--484, 1982.

\bibitem{Kato1949}
Tosio Kato.
\newblock On the upper and lower bounds of eigenvalues.
\newblock {\em Journal of the Physical Society of Japan}, 4:334--339, 1949.

\bibitem{Kato1995}
Tosio Kato.
\newblock {\em Perturbation Theory for Linear Operators}.
\newblock Classics in Mathematics. Springer, Berlin, 1995.
\newblock Reprint of the corrected printing of the second edition (1980).

\bibitem{KW1992}
J.~Kuczy\'{n}ski and H.~Wo\'{z}niakowski.
\newblock Estimating the largest eigenvalue by the power and {Lanczos}
  algorithms with a random start.
\newblock {\em SIAM Journal on Matrix Analysis and Applications},
  13(4):1094--1122, 1992.

\bibitem{LangerEtAl2001}
Heinz Langer, Alexander Markus, Vladimir Matsaev, and Christiane Tretter.
\newblock A new concept for block operator matrices: the quadratic numerical
  range.
\newblock {\em Linear Algebra and its Applications}, 330:89--112, 2001.

\bibitem{Lehmann1963}
N.~J. Lehmann.
\newblock Optimale {Eigenwerteinschlie{\ss}ungen}.
\newblock {\em Numerische Mathematik}, 5:246--272, 1963.

\bibitem{LumerRosenblum1959}
G{\"u}nter Lumer and Marvin Rosenblum.
\newblock Linear operator equations.
\newblock {\em Proceedings of the American Mathematical Society}, 10(1):32--41,
  1959.

\bibitem{MatsumotoTsuda1983}
K.~Matsumoto and I.~Tsuda.
\newblock Noise-induced order.
\newblock {\em Journal of Statistical Physics}, 31(1):87--106, 1983.

\bibitem{Miyajima2014a}
Shinya Miyajima.
\newblock Fast enclosure for all eigenvalues and invariant subspaces in
  generalized eigenvalue problems.
\newblock {\em SIAM Journal on Matrix Analysis and Applications},
  35(3):1205--1225, 2014.

\bibitem{Miyajima2014}
Shinya Miyajima.
\newblock Verified bounds for all the singular values of matrix.
\newblock {\em Japan Journal of Industrial and Applied Mathematics},
  31(3):513--539, 2014.

\bibitem{Murray2026}
Toby Murray.
\newblock {Lipschitz}-based robustness certification under floating-point
  execution, 2026.
\newblock arXiv:2603.13334.

\bibitem{Naumov2025}
Alexey Naumov, Maxim Rakhuba, Denis Ryapolov, and Sergey Samsonov.
\newblock On the upper bounds for the matrix spectral norm, 2025.
\newblock arXiv:2506.15660.

\bibitem{Nisoli2026}
Isaia Nisoli.
\newblock Certified spectral approximation of transfer operators and the
  {Gauss} map, 2026.
\newblock arXiv:2602.19435.

\bibitem{NisoliData2026}
Isaia Nisoli.
\newblock Data for ``{The} reverse {Ritz} inclusion: computable bounds on the
  pseudospectra of an operator by those of a finite-rank compression''.
\newblock Harvard Dataverse, 2026.

\bibitem{RigorousPseudospectra2026}
Isaia Nisoli.
\newblock {RigorousPseudospectra}.jl.
\newblock
  \url{https://github.com/JuliaBallArithmetic/RigorousPseudospectra.jl}, 2026.
\newblock Tag \texttt{reverse-ritz-arxiv-v2}, commit \texttt{c1c6580}.

\bibitem{RigorousInvariantMeasures2024}
Isaia Nisoli and Federico Poloni.
\newblock {RigorousInvariantMeasures}.jl.
\newblock \url{https://github.com/JuliaDynamics/RigorousInvariantMeasures.jl},
  2026.
\newblock Branch \texttt{ballarithmetic-0.4}, commit \texttt{ff173e7}.

\bibitem{Nussbaum1970}
Roger~D. Nussbaum.
\newblock The radius of the essential spectrum.
\newblock {\em Duke Mathematical Journal}, 37(3):473--478, 1970.

\bibitem{Oishi2001}
Shin'ichi Oishi.
\newblock Fast enclosure of matrix eigenvalues and singular values via rounding
  mode controlled computation.
\newblock {\em Linear Algebra and its Applications}, 324:133--146, 2001.

\bibitem{Plum1991}
Michael Plum.
\newblock Bounds for eigenvalues of second-order elliptic differential
  operators.
\newblock {\em Zeitschrift f\"ur angewandte Mathematik und Physik},
  42(6):848--863, 1991.

\bibitem{Rump2011}
Siegfried~M. Rump.
\newblock Verified bounds for singular values, in particular for the spectral
  norm of a matrix and its inverse.
\newblock {\em BIT Numerical Mathematics}, 51(2):367--384, 2011.

\bibitem{Rump2022a}
Siegfried~M. Rump.
\newblock Verified error bounds for all eigenvalues and eigenvectors of a
  matrix.
\newblock {\em SIAM Journal on Matrix Analysis and Applications},
  43(4):1736--1754, 2022.

\bibitem{RumpLange2023}
Siegfried~M. Rump and Marko Lange.
\newblock Fast computation of error bounds for all eigenpairs of a {Hermitian}
  and all singular pairs of a rectangular matrix with emphasis on eigen- and
  singular value clusters.
\newblock {\em Journal of Computational and Applied Mathematics}, 434:115332,
  2023.

\bibitem{Saad2011}
Yousef Saad.
\newblock {\em Numerical Methods for Large Eigenvalue Problems}.
\newblock Society for Industrial and Applied Mathematics, 2011.

\bibitem{SakuraiSugiura2003}
Tetsuya Sakurai and Hiroshi Sugiura.
\newblock A projection method for generalized eigenvalue problems using
  numerical integration.
\newblock {\em Journal of Computational and Applied Mathematics}, 159:119--128,
  2003.

\bibitem{SarichEtAl2010}
Marco Sarich, Frank No\'e, and Christof Sch\"utte.
\newblock On the approximation quality of {Markov} state models.
\newblock {\em Multiscale Modeling \& Simulation}, 8(4):1154--1177, 2010.

\bibitem{SatoNisoli2026}
Yuzuru Sato and Isaia Nisoli.
\newblock Statistical periodicity in noise-induced order from
  {Ruelle--Pollicott} resonances, 2026.
\newblock arXiv:2607.18771.

\bibitem{Stewart1971}
G.~W. Stewart.
\newblock Error bounds for approximate invariant subspaces of closed linear
  operators.
\newblock {\em SIAM Journal on Numerical Analysis}, 8(4):796--808, 1971.

\bibitem{Stewart1973}
G.~W. Stewart.
\newblock Error and perturbation bounds for subspaces associated with certain
  eigenvalue problems.
\newblock {\em SIAM Review}, 15(4):727--764, 1973.

\bibitem{StewartSun1990}
G.~W. Stewart and Ji-guang Sun.
\newblock {\em Matrix Perturbation Theory}.
\newblock Academic Press, Boston, 1990.

\bibitem{SunEtAl2021}
Xiaoming Sun, David~P. Woodruff, Guang Yang, and Jialin Zhang.
\newblock Querying a matrix through matrix-vector products.
\newblock {\em ACM Transactions on Algorithms}, 17(4):1--19, 2021.

\bibitem{Temple1928}
George Temple.
\newblock The theory of {Rayleigh's} principle as applied to continuous
  systems.
\newblock {\em Proceedings of the Royal Society A}, 119(782):276--293, 1928.

\bibitem{TeraoEtAl2025}
Takeshi Terao, Yoshitaka Watanabe, and Katsuhisa Ozaki.
\newblock Verified error bounds for the singular values of structured matrices
  with applications to computer-assisted proofs for differential equations,
  2025.
\newblock arXiv:2502.09984.

\bibitem{Tobler2025}
James Tobler, Hira~Taqdees Syeda, and Toby Murray.
\newblock A formally verified robustness certifier for neural networks.
\newblock In Ruzica Piskac and Zvonimir Rakamari\'{c}, editors, {\em Computer
  Aided Verification (CAV 2025)}, Lecture Notes in Computer Science, pages
  327--348. Springer, 2025.

\bibitem{TrefethenEmbree2005}
Lloyd~N. Trefethen and Mark Embree.
\newblock {\em Spectra and Pseudospectra: The Behavior of Nonnormal Matrices
  and Operators}.
\newblock Princeton University Press, 2005.

\bibitem{Tretter2008}
Christiane Tretter.
\newblock {\em Spectral Theory of Block Operator Matrices and Applications}.
\newblock Imperial College Press, London, 2008.

\bibitem{Tucker2011}
Warwick Tucker.
\newblock {\em Validated Numerics: A Short Introduction to Rigorous
  Computations}.
\newblock Princeton University Press, Princeton, NJ, 2011.

\bibitem{Varga2004}
Richard~S. Varga.
\newblock {\em Ger{\v s}gorin and His Circles}, volume~36 of {\em Springer
  Series in Computational Mathematics}.
\newblock Springer, Berlin, Heidelberg, 2004.

\bibitem{WilliamsEtAl2015}
Matthew~O. Williams, Ioannis~G. Kevrekidis, and Clarence~W. Rowley.
\newblock A data-driven approximation of the {K}oopman operator: Extending
  dynamic mode decomposition.
\newblock {\em Journal of Nonlinear Science}, 25(6):1307--1346, 2015.

\end{thebibliography}
\endgroup

\end{document}